\documentclass{amsart}
\usepackage{lmodern}
\usepackage[utf8]{inputenc}
\usepackage[T1]{fontenc}
\usepackage[english]{babel}
\usepackage[hyperref,backend=biber, style=alphabetic,giveninits]{biblatex}

\usepackage{xcolor}
\usepackage{amsmath,amssymb,amsthm,graphicx,amsfonts,enumerate,stmaryrd, mathrsfs,dsfont}
\usepackage{enumitem, tikz, xcolor}
\usepackage[foot]{amsaddr}
\usepackage[all]{xy}
\usepackage[toc,page]{appendix}

\definecolor{bleu_sombre}{rgb}{0,0,0.6}  
\definecolor{rouge_sombre}{rgb}{0.8,0,0}\definecolor{vert_sombre}{rgb}{0,0.6,0}
\usepackage[plainpages=false,colorlinks,linkcolor=bleu_sombre,
citecolor=rouge_sombre,urlcolor=vert_sombre,breaklinks]{hyperref}

\theoremstyle{plain}
\newtheorem{theorem}{{Theorem}}[section]
\newtheorem*{theorem*}{{Theorem}}
\newtheorem{proposition}[theorem]{Proposition}

\newtheorem*{proposition*}{Proposition}
\newtheorem{corollary}[theorem]{Corollary}
\newtheorem*{corollary*}{Corollary}
\newtheorem{lemma}[theorem]{Lemma}

\newtheorem*{lemma*}{Lemma}

\theoremstyle{definition}
\newtheorem{definition}[theorem]{Definition}
\newtheorem*{definition*}{Definition}

\theoremstyle{remark}
\newtheorem{remark}[theorem]{Remark}
\newtheorem{notation}[theorem]{Notation}

\newcommand\reallywidecheck[1]{%
\savestack{\tmpbox}{\stretchto{%
  \scaleto{%
    \scalerel*[\widthof{\ensuremath{#1}}]{\kern-.6pt\bigwedge\kern-.6pt}%
    {\rule[-\textheight/2]{1ex}{\textheight}}
  }{\textheight}%
}{0.5ex}}%
\stackon[1pt]{#1}{\scalebox{-1}{\tmpbox}}%
}

\newcommand{\ep}{\epsilon}

\newcommand{\R}{\mathbb{R}}
\newcommand{\C}{\mathbb{C}}

\newcommand{\N}{\mathbb{N}}
\newcommand{\Z}{\mathbb{Z}}

\newcommand{\op}{\mathrm{Op}}

\newcommand{\dd}[1]{\ensuremath{\operatorname{d}\!{#1}}}

\newcommand{\flechebas}[1]{%
  \settoheight{\unitlength}{\mbox{$#1$}}%
  \settowidth{\Taille}{\mbox{~${\scriptstyle #1}$}}%
  \addtolength{\unitlength}{4ex}%
  \begin{picture}(0,1)
    \put(0,1){\vector(0,-1){1}}
    \put(0,0.5){\makebox(0,0){${\scriptstyle #1}$ \hspace{\the\Taille}}}
  \end{picture}}
\newcommand{\flechehaut}[1]{%
  \settoheight{\unitlength}{\mbox{$#1$}}%
  \settowidth{\Taille}{\mbox{~${\scriptstyle #1}$}}%
  \addtolength{\unitlength}{4ex}%
  \begin{picture}(0,1)
    \put(0,0){\vector(0,1){1}}
    \put(0,0.5){\makebox(0,0){\hspace{\the\Taille}${\scriptstyle #1}$ }}
  \end{picture}}
\newcommand{\flechedroite}[1]{%
  \settowidth{\unitlength}{\mbox{$#1$}}
  \settoheight{\Taille}{\mbox{${\scriptstyle #1}$}}
  \addtolength{\Taille}{1ex}
  \addtolength{\unitlength}{4ex}
  \raisebox{0.5ex}{%
  \begin{picture}(1,0)
    \put(0,0){\vector(1,0){1}}
    \put(0.5,0){\makebox(0,0){${\scriptstyle #1}$ \vspace{\the\Taille}}}
  \end{picture}}}
\newcommand{\flechegauche}[1]{%
  \settowidth{\unitlength}{\mbox{$#1$}}
  \settoheight{\Taille}{\mbox{${\scriptstyle #1}$}}
  \addtolength{\Taille}{1ex}
  \addtolength{\unitlength}{4ex}
  \begin{picture}(1,0)
    \put(1,0){\vector(-1,0){1}}
    \put(0.5,0){\makebox(0,0){${\scriptstyle #1}$ \vspace{\the\Taille}}}
  \end{picture}}

\newcommand{\cqfd}{\hfill $\square$\par\vspace{1ex}}

\newcommand{\h}{\hbar}

\newcommand{\T}{\mathbb{T}}

\renewcommand{\O}{\mathcal{O}}

\renewcommand{\geq}{\geqslant}
\renewcommand{\leq}{\leqslant}

\newcommand{\demons}[1][$\!\!$]{\noindent\textbf{\demon\ }\textsl{#1}\textbf{.}~}

\numberwithin{equation}{section}

\newcounter{hypo}
\def\thehypo{\textbf{\textup{(\h\arabic{hypo})}}}

\newcounter{listcont}
\def\thelistcont{\arabic{listcont}}

\title{Analytic Microlocal Bohr-Sommerfeld Expansions}
\author{Antide~\textsc{Duraffour}$^*$ ~} 
\address{$~^*$Univ. Rennes, CNRS, IRMAR -
    UMR 6625, F-35000 Rennes, France} 
   
\begin{document}
\maketitle
\begin{abstract}
This article is devoted to analytic-Gevrey estimates in $\hbar$, of the Bohr-Sommerfeld expansion of the eigenvalues of self-adjoint pseudo-differential operators $P := p^w(x,\h D_x)$ acting on $L^2(\R)$ in the regular case. We consider an interval of energies $(E_1,E_2)$ in which the spectrum of $P$ is discrete and such that the energy sets $\{ p(x,\xi) = E \}$ are compact regular connected curves. Under some assumptions on the holomorphy of the symbol $p$, we will use the isometry between $L^2(\R)$ and the Bargmann space to prove that the spectrum $\sigma_\h(P) \cap (E_1,E_2)$ of $p^w(x,hD_x)$ is described by equations of the form
\begin{equation}
\int_{\mathrm{orbit}} \xi dx + \pi \hbar +  c_2(\lambda) \h^2 + \cdots \in  2\pi \hbar \Z  .
\end{equation}
More precisely it is possible to build exponentially sharp quasimodes in the Bargmann space. This will allow us to describe the ``Bohr-Sommerfeld'' spectrum with precision $\mathcal{O}(e^{-\mathscr{C}/\h})$ when $\hbar \longrightarrow 0^+$ for some small $\mathscr{C}>0$. A precise examination of the principal symbols will provide an interpretation to the Maslov correction $\pi \hbar$ in the Bargmann space.
\end{abstract}

\section{Introduction}

The \emph{Bohr-Sommerfeld} rules are one of the oldest ``quantization''
results for obtaining quantum levels from the Hamiltonian of classical
mechanics. They appeared with Bohr's model postulating that the orbits of the electrons in the hydrogen atoms must satisfy the quantization condition
\begin{equation}
\frac{1}{2 \pi}\oint_{\mathrm{orbit}} \xi dx \in \h \Z.
\end{equation}
They were extended by Sommerfeld along elliptic orbits and made more precise by adding Maslov's correction \cite{Maslov} replacing $\hbar \Z$ by $\hbar (\Z+\frac12)$ in the previous equation. In mathematical terms, these
rules are expected to give the eigenvalues of a (pseudo-)differential
operator in terms of the geometry of its principal symbol.  Except in
very specific cases like the Harmonic oscillator, these
``Bohr-Sommerfeld eigenvalues'' are only approximations of the true
eigenvalues, with an error term that tends to zero as the
semiclassical parameter $\h$ tends to zero. Historically, the Bohr-Sommerfeld rules were given for Schrödinger operators:
\[
P = -\h^2\frac{\dd{}^2}{\dd{x}^2} + V(x).
\]
They describe regular eigenvalues, that is discrete eigenvalues in intervals without critical value of $V$, with an error term of order
$\O(\h^2)$, but it has been proved by Helffer and Robert \cite{HR84} that these rules can be extended
to 1D pseudo-differential operators
$P=p^w(x,\h D_x)$ (denoting $\h D_x := \frac{\h}{i}\frac{\dd{}}{\dd x}$) given by the Weyl quantization
\begin{equation}
\forall u \in \mathscr{S}(\R), \quad p^w(x,\h D_x)u := \frac{1}{2\pi \h}\iint_{\R^2} p\left( \frac{x+y}{2},\xi\right )e^{i\xi(x-y)/\h} u(y)d\xi dy.
\end{equation}

They give asymptotic expansions for eigenvalues that are exact to any order\footnote{We say that the precision is $\mathcal{O}(\h^{\infty})$.} $\O(\h^N)$, 
$N>0$. In the region where the spectrum is discrete and the energy surfaces $E_{\lambda} := \{ (x,\xi) \in \R^2 ~ | ~ p(x,\xi) =\lambda) \}$ are regular connected curves, the eigenvalues admit a full asymptotic expansion starting as follows, for some $k \in \Z$
\begin{equation}
\label{operator}
\lambda_k = A^{-1}(2 \pi (k-1/2) \h) + \O(\h^2) \hbox{ where } A(\lambda) := \int_{E_{\lambda}} \xi dx.
\end{equation}

The Bohr-Sommerfeld rules are microlocal in essence as they are based upon quantities defined in the quantum phase-space. They do not
necessarily describe the entire spectrum of the operator $P$, but only
those eigenvalues that are associated with a regular, connected
component of a level set of the classical Hamiltonian $p(x,\xi)$. The full spectral asymptotics, i.e. approximation with sharpness $\mathcal{O}(\h^\infty)$ of these eigenvalues, were tackled in many articles. For instance, they can be found in Rozenblum's article \cite{Rozenblum75} in the case of the circle or in Helffer-Robert \cite{HR84} for a much broader presentation.

In scenarios of level sets containing several connected components, the approximate
spectrum is obtained by taking ``the union'' of the Bohr-Sommerfeld
eigenvalues of each component (although this statement is difficult to
locate in the literature). From the viewpoint of quantum mechanics,
separate components can actually interact, due to
\emph{tunnelling}. However, the impact of tunnelling on eigenvalues is
expected to be exponentially small: $\O(e^{-\mathscr{C}/\h})$ and does not
affect the asymptotic expansion. An interesting situation occurs when
the Bohr-Sommerfeld expansions for two different components are
identical; this occurs, for instance, for a Schrödinger operator with
a symmetric double well. The usual Bohr-Sommerfeld expansion proves
that the spectrum contains doublets of eigenvalues that are
$\O(\h^\infty)$-close to each other. However, it cannot prove the
exponentially small splitting obtained by Helffer-Sjöstrand \cite{HS84} in the case of the Schrödinger operator $-\hbar^2 \Delta + V(x)$, and hence fails to confirm an aspect of the quantum
tunelling. 

The goal of this paper is to reach an exponentially small $\mathcal{O}(e^{-\frac{\mathscr{C}}{\hbar}})$ but not optimal
precision in the case of a single connected component in a broader setting than Schrödinger operators. In principle, as recalled in \cite[Section 8]{HI} the delicate construction of \cite{GeSj}, concerning a situation where semiclassical resonances appear, can be applied to prove the result. Hence it is considered, at least morally, known by experts of the field. However, its interest regarding microlocal quantum tunneling warrants an accessible proof that we hope this article provides. In particular, the author also wishes that, by the end of this article, the reader will be convinced that this result generalizes naturally to the case where the energy sets are composed of several connected components: the spectrum obtained is ``the union'' of that associated with each connected component. However, in order not to overload the discussion, the focus of this article is the single component case.  

Let us finally mention that the recent work of Hitrik-Zworski \cite{hitrik2024} gives an analogue of Theorem \ref{th1} in the case of an elliptic singularity, in both self adjoint and non self adjoint settings. This refines and adds details to the results of \cite[Section 3.5]{HII} and \cite[Appendix b]{HIII} where the authors consider, in a self adjoint setting, the case of an elliptic critical point. Note also that the case of a saddle critical point is also treated in \cite[Appendix b]{HIII}. 

\subsection{Mathematical Formulation of the Problem}
In order to be general enough, we will consider unbounded pseudo-differential operators defined on dense domains in $L^2(\R)$. We will also need to introduce the holomorphic framework that will allow us to get exponentially accurate estimates. 

\subsubsection{Assumptions on the symbol}
We start with the definition of a \emph{tubular neighbourhood of a vector space}. This will be useful for introducing the holomorphic setting.

\begin{definition}[Tubular Neighbourhood]
Letting $d \in \N^*$, a tubular neighbourhood of a real subspace $S \in \C^d$, is a small open neighbourhood of $S$ in $\C^d$, satisfying $U + S = U$ and $U \subset S + \mathbb{B}_{\C^d}(0,r_0)$ for some $r_0 > 0$. 
\end{definition}

Note also that in Appendix \ref{apx.1} we have recalled the main results about \emph{classical analytic symbols} developped in \cite{AST_1982__95__R3_0} necessary for understanding the formulation of the main Theorem \ref{th1}. 

We consider a pseudo-differential operator $P: \mathscr{S}(\R) \longrightarrow \mathscr{S}(\R)$ (extended to a domain $\mathscr{D} \subset L^2(\R)$) given by the Weyl quantization 
\begin{equation}
\forall u \in \mathscr{S}(\R), ~~ Pu:=\frac{1}{2\pi \hbar}\iint_{\R^2} e^{\frac{i}{\hbar}(x-y)\xi}p(\tfrac{x+y}{2},\xi)u(y)dy d\xi,
\end{equation}
 of a symbol $p$ whose class will be defined down below. 
We use order functions $\mathfrak{m} \in \mathscr{C}^{\infty}(\R^2,\R_+^*)$ defined in \cite[Chapter 4.4]{Zworski} or \cite[Chapter 2]{martinez_andre_introduction_2002} which satisfies $\liminf_{|X| \to + \infty} \mathfrak{m}(X) > 0$ and for some $N \in \N$,
\begin{equation}
\label{eq:order}
\exists N \in \N, ~ \exists C >0, ~ \forall (X,\widehat{X}) \in \R^2 \times \R^2, \quad \mathfrak{m}(X) \leq C \langle X-\widehat{X} \rangle^N \mathfrak{m}(\widehat{X}), 
\end{equation}
where $\langle (x,\xi) \rangle = (1+x^2 + \xi^2)^\frac12$ is the usual Japanese bracket. We will see that the property \eqref{eq:order} will be useful when estimating norms of kernel operators with Schur's Lemma. In $\C^2$, we extend this by requiring \eqref{eq:order} with the notation
\[\forall (x,\xi)\in \C^2, \quad \langle (x,\xi) \rangle = \langle (\mathrm{Re}(x),\mathrm{Re}(\xi)) \rangle \hbox{ and } \mathfrak{m}(x,\xi) :=  \mathfrak{m}(\mathrm{Re}(x), \mathrm{Re}(\xi)).\]
 We let $U_0$ be a tubular neighbourhood of $\R^2$ and define
\begin{equation}
B(U_0,\mathfrak{m}) := \{  u \in \mathrm{Hol}(U_0) ~ | ~ \exists C>0, ~ |u| \leq C \mathfrak{m} \}.
\end{equation}
The reader who is used to symbols classes in the smooth setting will notice the following.
\begin{remark}
Thanks to the Cauchy estimates, letting $U \subset U_0$ another tubular neighbourhood of $\R^2$, satisfying $\overline{U} \subset U_0$, the set $B(U_0,\mathfrak{m})$ is contained in $S(U,\mathfrak{m})$ defined by
\begin{equation}
S(U,\mathfrak{m}) = \{ u \in \mathscr{C}^{\infty}(U) ~ | ~ \forall \alpha \in \N, ~~ |\partial^{\alpha} u| \leq C_\alpha \mathfrak{m} \}.
\end{equation}
\end{remark}
Let us consider an interval of energies $(E_1,E_2)$, a symbol $p \in B(U_0,\mathfrak{m})$. We suppose that for some $\delta >0$, the following assumptions \ref{enum:1}, \ref{enum:2}, \ref{enum:3} hold.
\begin{enumerate}[label= \textit{H.\arabic*}, ref=\textit{H.\arabic*}]
\item \label{enum:1} \textit{Localisation condition}: $p^{-1}((-\infty,E_2+\delta])$ is compact and non empty,
\item \label{enum:2} \textit{Description of the energy sets}: $\nabla p \neq 0$ on $p^{-1}((E_1-\delta,E_2+\delta))$ and the energy curves $E_\lambda := \{p = \lambda \}$ are nonempty and connected for $\lambda \in (E_1-\delta,E_2+\delta)$.
\item \label{enum:3} \textit{Ellipticity}: There exists $C,C_0>0$ such that $\tfrac{1}{C_0} \mathfrak{m} \leq C + p \leq C_0 \mathfrak{m}$ on $U_0$.
\end{enumerate} 
In particular Assumptions \ref{enum:1},\ref{enum:3} imply that $p-(E_2+\delta/2)$ is elliptic in the class $S(U_0,\mathfrak{m})$ in the sense that
\begin{equation}
\label{eq:ellip}
\exists c_0,R_0>0, \quad \forall (x,\xi) \in U_0, ~ |(x,\xi)| \geq R_0, \quad p(x,\xi)-(E_2+\tfrac{\delta}{2}) \geq c_0 \mathfrak{m}(x,\xi). 
\end{equation}
Letting $K:= p^{-1}((-\infty,E_2+\delta])$ compact because of Assumption \ref{enum:1}, then \eqref{eq:ellip} comes from a linear combination of the two inequalities $p-(E_2+\tfrac{\delta}{2}) \geq \frac{1}{C_0}\mathfrak{m}-C$ and $p-(E_2+\tfrac{\delta}{2})\geq \tfrac{\delta}{2}$ on $K^c$ (add the second multiplied by $\tfrac{2C}{\delta}$  to the first).

\begin{remark}
In the case where $\liminf_{|(x,\xi)| \to + \infty} \mathfrak{m} = + \infty$ Assumption \ref{enum:1} is a consequence of \ref{enum:3} if we assume $p^{-1}((-\infty,E_2+\delta])$ is non empty.
\end{remark}

By construction the operator $P$ is symmetric and defined on the domain $\mathscr{S}(\R)$. Thanks to the usual results on pseudo-differential operators, since on $\R^2$,
 $|p-i| \gtrsim c\mathfrak{m} $ for some $c>0$ and $p-i \in S(\R^2,\mathfrak{m})$, $P-i$ is invertible for $\hbar$ small enough. This proves that $(P,\mathscr{S}(\R))$ admits a unique self-adjoint extension (for $\hbar$ small enough) that we will denote $(P,\mathscr{D}(P))$. By Fredholm theory (for instance see \cite[Appendix C]{Dyatlov}), the fact that $p-(E_2+\tfrac{\delta}{2})$ is elliptic at infinity in $S(\R^2,\mathfrak{m})$ implies that $P$ has discrete spectrum in $(-\infty,E_2+\tfrac{\delta}{2})$.

\subsubsection{Statement of the Main Theorem}
We will see in Section \ref{sec:chap3_geomana} that in this setting, the action integral $A(\lambda)=\int_{\{ p = \lambda \}\cap \R^2} \xi dx$ is well defined for $\lambda \in (E_1-\delta, E_2+\delta)$ and invertible as a function of $\lambda$. In the following theorem, "a function $\mathcal{O}(e^{-\frac{\mathscr{C}}{\hbar}})$" refers to a positive function $g : (0,\hbar_0) \longmapsto g(\hbar)$ satisfying 
\[ \exists R>0, ~ \forall \hbar \in (0,\hbar_0),  \quad g(\hbar) \leq Re^{-\frac{\mathscr{C}}{\hbar}}. \]

\begin{theorem}
\label{th1}
There exist a classical analytic symbol 
$
\lambda(\h,I) = A^{-1}(I-\pi \hbar) + \mathcal{O}(\hbar^2) 
$ defined for $I$ in a small neighbourhood of $A([E_1,E_2])$, 
a constant $\mathscr{C} > 0$ and a function $\mathcal{O}(e^{-\frac{\mathscr{C}}{\hbar}})$ such that, for $\h$ small enough
\begin{equation}
\label{th1eq}
\begin{aligned}
& \forall \lambda_\h \in \sigma_\h(P)  \cap (E_1,E_2), ~ \exists ! k \in \Z, ~~ |\lambda_\h - \lambda(\h,2\pi k \h)| \leq \mathcal{O}(e^{-\frac{\mathscr{C}}{\hbar}}), \\
& \forall k \in \Z, ~  \hbox{ s.t. } \lambda(\h,2\pi k \h) \in (E_1,E_2), \quad \exists ! \lambda_\h \in \sigma_\h(P), ~~ |\lambda_\h - \lambda(\h,2 \pi k \h)| \leq \mathcal{O}(e^{-\frac{\mathscr{C}}{\hbar}}),
\end{aligned}
\end{equation}
and the eigenvalues of $P$ in $(E_1,E_2)$ are simple.
\end{theorem}

Paying the price of a small imprecision, it is more natural to remember that we want to prove that $\sigma_\hbar(P) \cap (E_1,E_2)$ is 
\begin{align}
\left \{ \lambda(\h,2\pi k\h) + \mathcal{O}(e^{-\mathscr{C}/\h})~ | ~ k\in \Z \right \} \cap (E_1,E_2).
\end{align}

In the following sections, we show that for each $\lambda_0 \in \mathrm{Neigh}([E_1,E_2],\R)$ Theorem \ref{th1} stands replacing $(E_1,E_2)$ with some small open neighbourhood of $\lambda_0$. By the patching Lemma \ref{lem.patching} the result will extend on the whole segment $[E_1,E_2]$.

\subsection{Scheme of the Proof} \hfill \\
In the first section we follow \cite{lres} or \cite[Chapter 13]{Zworski} and conjugate $p^w(x,\h D_x)$ with the Bargmann transform into a new pseudo-differential operator $p_{\Phi_0}(z,\h D_z)$ acting on 
the Bargmann space. We then use Sjöstrand's result obtained in \cite[Chapter 4]{AST_1982__95__R3_0} to build analytic WKB expansions around points of the complex energy curves. In the end, we show that the condition on $\lambda$ for which we can merge these local WKB expansions (using a $\overline{\partial}$ lemma) into a full quasimode yields the Bohr-Sommerfeld quantization condition with exponential sharpness.

\section{Bargmann Transform and complex pseudo-differential operators}
 \textit{In this section we introduce some aspects of the theory of the Bargmann space and complex analytic pseudo-differential operators. After conjugating $P$ with the Bargmann transform, following Sj\"ostrand in \cite[Chapter 3]{AST_1982__95__R3_0} we write the operator as an analytic pseudo-differential operator using integrals on ``good contours''.}

\subsection{The Bargmann Transform} \hfill \\

We define the spaces of weighted holomorphic functions, see for instance Zworski's book \cite[Chapter 13]{Zworski}. Here we will only use the weight associated with the Bargmann transform
\begin{equation}
\Phi_0(z) = \frac{1}{2}|z|^2.
\end{equation}
We define the global spaces as follows, denoting $\displaystyle \mathrm{d}L(z) = \frac{\mathrm{d}\overline{z} \wedge \mathrm{d}z}{2i}$ the Lebesgue measure,
\begin{equation}
\begin{aligned}
& L^2_{\Phi_0}(\C) := \left \{ f : \C \longrightarrow \C \hbox{ measurable } ~  \Bigl | ~  \int_{\C} |f(z)|^2e^{-2 \Phi_0(z)/\h} \mathrm{d}L(z) < \infty \right \}, \\
& L^2_{\Phi_0}(\C,\mathfrak{m}) := \left \{ f : \C \longrightarrow \C \hbox{ measurable } ~  \Bigl | ~  \int_{\C} |f(z)|^2 \mathfrak{m}(z)^2 e^{-2 \Phi_0(z)/\h} \mathrm{d}L(z) < \infty \right \},
\\
& H_{\Phi_0}(\C):=  L^2_{\Phi_0}(\C) \cap \mathrm{Hol}(\C), ~H_{\Phi_0}(\C,\mathfrak{m}):=  L^2_{\Phi_0}(\C,\mathfrak{m}) \cap \mathrm{Hol}(\C).
\end{aligned}
\end{equation}

The Bargmann transform $\mathcal{T}_0 : L^2(\R) \longrightarrow H_{\Phi_0}(\C)$  
defined by 
\begin{equation}
\forall v \in L^2(\R), ~~ \mathcal{T}_0 v(z) := \frac{1}{(\pi \h)^{3/4}}\int_{\R} e^{\frac{i}{\hbar}\varphi_0(z,y)} v(y)dy, \quad \varphi_0(z,y) = i\left(\tfrac{z^2}{2}-\sqrt{2}yz+\tfrac{y^2}{2}\right),
\end{equation}
 is a unitary transform (see \cite[Theorem 13.7]{Zworski}). It is associated with the symplectomorphism 
\begin{align}
\kappa_0 : \left\{ \begin{array}{ccc}
\R^{2} & \longrightarrow & \Lambda_{\Phi_0} \subset \C^2 \\
(x,\xi) & \longmapsto & \frac{1}{\sqrt{2}}(x-i\xi,\xi-ix)
\end{array} \right. 
\end{align}
where \[ \begin{split}
 \Lambda_{\Phi_0} := \{ (x-i\xi,\xi-ix) ~ | ~ (x,\xi) \in \R^2 \} = \{(z,-i \overline{z}) ~ | ~ z \in \C \} \\ = \left \{ \left (z,\tfrac{2}{i} \partial_z \Phi_0(z) \right ) ~ | ~ z \in \C \right \},
\end{split} \] plays the role of ``phase space'' in our new representation of $L^2(\R)$ functions. We will denote (with a slight abuse of notations) $\mathfrak{m} \circ \kappa_0^{-1}$ simply by $\mathfrak{m}$


 Let us recall without proof a result of Sjöstrand (see the lecture notes \cite[Section 12.2]{lres} or Zworski's book \cite[Theorem 13.9]{Zworski}) that is the starting point of our analysis. In the following theorem, 
 \begin{equation}
 S(\Lambda_{\Phi_0},\mathfrak{m}) := \kappa_{0,*}(S(\R^2,\mathfrak{m}))= \{ v \in \mathscr{C}^{\infty}(\Lambda_{\Phi_0}) ~ | ~ |\partial^{\alpha} u|  <  C_{\alpha}\mathfrak{m}, ~ \forall \alpha \in \N^2 \}.
 \end{equation}
 


\begin{theorem}[Real Weyl Quantization on $H_{\Phi_0}(\C)$] \hfill \\
\label{wquant}
Let $a \in S(\R^2,\mathfrak{m})$ then $a^w(x,\h D_x)$ with domain $\mathscr{S}(\R)$ is unitarily equivalent via $\mathcal{T}_0$, to the operator
$a_{\Phi_0}^w(z,\h D_z) = \mathcal{T}_0 A \mathcal{T}_0^* : e^{\Phi_0/\hbar}\mathscr{S}(\C)\cap \mathrm{Hol}(\C) \longrightarrow H_{\Phi_0}(\C)$
given by 
\begin{align}
\label{pseudo}
a_{\Phi_0}^w(z,\h D_z)u = \frac{1}{2\pi \h}\iint_{\Lambda_{0}(z)} a_{\Phi_0}\left( \tfrac{z+w}{2},\zeta \right) e^{\frac{i}{\h}\zeta(z-w)}u(w)\mathrm{d}\zeta \wedge \mathrm{d}w,
\end{align}
where $\Lambda_0(z)$ is the contour defined with $\zeta = \tfrac{2}{i}\partial_z \Phi_0 \left( \frac{z+w}{2} \right)$ and $a_{\Phi_0} = a \circ \kappa_0^{-1}$.
\end{theorem}

We will often denote $P_{\Phi_0}$ for $\mathcal{T}_0 P \mathcal{T}_0^*$ instead of $p_{\Phi_0}(z,\h D_z)$ to lighten the writings. Since $p_{\Phi_0}(z,\zeta) = p(\frac{1}{\sqrt{2}}(z+i\zeta,\zeta+iz))$, $p_{\Phi_0}$ can be extended by holomorphy to an element of $\mathcal{B}(\mathcal{U}_0,\mathfrak{m})$ where $\mathcal{U}_0 = \kappa_0(U_0)$ is a tubular neighbourhood of $\Lambda_{\Phi_0}$ and the definition of $\mathcal{B}(\mathcal{U}_0,\mathfrak{m})$ is
\begin{equation}
 \mathcal{B}(\mathcal{U}_0,\mathfrak{m}) = \{ u \in \mathrm{Hol}(\mathcal{U}_0) ~ | ~ |u(z,\zeta)|  < C \mathfrak{m}(z) \}
 \end{equation}
Since $\mathcal{T}_0$ is unitary, $P_{\Phi_0}$ with domain $e^{\frac{\Phi_0}{\hbar}}\mathscr{S}(\C)\cap \mathrm{Hol}(\C)$ has a unique self-adjoint extension that we still denote $P_{\Phi_0}$ (with domain $\mathscr{D}(P) := H(\mathfrak{m})$ the Sobolev space associated with the order function $m$). This extension is uniquely defined on the domain $\mathscr{D}(P_{\Phi_0})=\mathcal{T}_0 \mathscr{D}(P)$ containing $H_{\Phi_0}(\C,\mathfrak{m})$ (in fact because of the ellipticity assumption \ref{enum:3} they are equal) and induces a continuous operator $H_{\Phi_0}(\C,\mathfrak{m}) \longrightarrow L^2_{\Phi_0}(\C)$.

Let us now show that we can bound the $L^{\infty}$ norms of the elements of $H_{\Phi_0}(\C)$ by their $L^2_{\Phi_0}(\C)$ norm. This can be interpreted as an obstacle for approaching Dirac distributions (physically representing classical particles) with functions in $L^2_{\Phi_0}(\C)$ (representing quantum states at fixed $\hbar$).

\begin{lemma}[Uncertainty Principle]
\label{linf} 
\begin{align}
\forall u_\h \in H_{\Phi_0}(\C), ~ \forall z \in \C, \quad | u_\h(z)| \leq \frac{2^\frac14}{\sqrt{\pi \h}}e^{\Phi_0(z)/\h} \lVert u_\h \rVert_{L^2_{\Phi_0}(\C)}.
\end{align}
\end{lemma}
\begin{proof}
We write $z = x+ i \xi$ we have 
$\mathrm{Re}(i \varphi_0(z,y)) = \frac{|z|^2}{2}-\frac{1}{2}(y-\sqrt{2}x)^2$, then denoting $v_\h = \mathcal{T}_0^* u_\h$ we have for all $z \in \C$,
\begin{align*}
|u_\h(z)| = \left | c_\h\int_{\R} e^{i \varphi_0(z,\cdot)/\h} v_\h(y)dy \right | & \leq e^{\Phi_0(z)/\hbar}\int_{\R} c_\hbar e^{-\frac{1}{2}(y-\sqrt{2}x)^2/\h} |v_\h(y)|dy \\
& \leq  e^{\Phi_0(z)/\hbar}\lVert c_\h e^{-\frac{1}{2}(y-\sqrt{2}x)^2/\h} \rVert_{L^2(\R)} \lVert v_\h \rVert_{L^2(\R)} \\
& \leq  e^{\Phi_0(z)/\hbar}\lVert c_\h e^{-\frac{1}{2}(y-\sqrt{2}x)^2/\h} \rVert_{L^2(\R)} \lVert u_\h \rVert_{L_{\Phi_0}^2(\C)}
\end{align*}
and $\lVert c_\h e^{-\frac{1}{2}(y-\sqrt{2}x)^2/\h} \rVert_{L^2(\R)} = \frac{2^\frac14}{\sqrt{\pi \h}}$.
\end{proof}

%
%

\subsection{Complex Pseudo-differential Operators}
 In the following lemma, we obtain an integral formula on a different contour (than that of \eqref{pseudo}) for $P_{\Phi_0}$. The benefit of such change of contour is that it can be shrunk to a local contour making only ``an exponentially small error''.

\begin{lemma}
\label{lem:gc}
Let $z \in \C$, for $c \in \R_+^*$, define the contour
\begin{equation}
\label{eq.contour}
\Lambda_c(z) =\left \{ (w,\zeta) \in \C^2 ~  | ~  \zeta = \tfrac{2}{i}\partial_z \Phi_0\left( \frac{z+w}{2} \right) + ic \frac{\overline{z-w}}{\langle z-w \rangle}\right \}, 
\end{equation}
and choose $c \in \R_+^*$ small enough so that $\Lambda_c \subset \mathcal{U}_0$. We have the \textbf{exact} formula, for all $z \in \C$
\begin{equation}
\label{eq:pseudogc}
\forall u_\h \in  H_{\Phi_0}(\C,\mathfrak{m}),  ~~ 
 P_{\Phi_0}u_\h(z) = \frac{1}{2\pi \h}\iint_{\Lambda_c(z)} p_{\Phi_0} \left( \frac{w+z}{2}, \zeta \right) e^{i\zeta(z-w)/\h} u_\h(w)d\zeta \wedge dw. 
\end{equation}
\end{lemma}
Notice that in particular, the contour $\Lambda_c$ will satisfy the estimate \ref{eq:convcont} and that the integral is converging. This lemma is proved in \cite[Formula 1.4.17]{minicourse}, here we briefly provide some elements of the proof which will be useful in the rest of the article. In the proof, for all $z \in \C$, we will denote $\mathfrak{m}(z,\tfrac{2}{i} \partial_z \Phi_0(z))$ simply by $\mathfrak{m}(z)$.
\begin{proof}
The distance between the contours $\Lambda_c$ and $\Lambda_{\Phi_0}$ must be bounded to remain in the domain where $p_{\Phi_0}$ is holomorphic, that is why we cannot consider the more natural contour $\zeta = \tfrac{2}{i} \partial_z \Phi_0\left(\frac{z+w}{2} \right) + ic\overline{(z-w)}$. We recall that for all $u_\h \in e^{\Phi_0/\h}\mathscr{S}(\C) \cap H_{\Phi_0}(\C)$, for all $N \in \N$, $|\mathfrak{m}(w)e^{-\Phi_0(w)/\hbar}u_\hbar(w)| \leq \dfrac{C_N}{\langle w \rangle^N}$, for some $C_N >0$. Using Assumption \ref{eq:order} on $\mathfrak{m}$, by holomorphic change of contour we obtain, for all $u_\h \in e^{\Phi_0/\h}\mathscr{S}(\C) \cap H_{\Phi_0}(\C)$
\begin{equation}
\label{eq.pseudolc}
P_{\Phi_0} u_\h(z) = \frac{1}{2\pi \hbar}\iint_{\Lambda_c(z)} p \left( \frac{w+z}{2}, \zeta \right) e^{i\zeta(z-w)/\h} u_\h(w)dw \wedge d\zeta.
\end{equation}

The contour $\Lambda_c(z)$ verifies
\begin{align*}
\forall (w,\zeta) \in \Lambda_c(z), \quad -\mathrm{Im}(\zeta(z-w)) + \Phi_0(w) = \Phi_0(z) -c \frac{|z-w|^2}{\langle z-w \rangle},
\end{align*} 
so that defining $\psi_z(w,\zeta) = -\mathrm{Im}(\zeta(z-w)) + \Phi_0(w)$ we have 
\begin{equation}
\label{eq:convcont}
\forall (w,\zeta) \in \Lambda_c, ~~ \psi_z(w,\zeta)- \psi_z(z,\tfrac{2}{i}\partial_z \Phi_0(z)) = -c \frac{|z-w|^2}{\langle z-w \rangle}.
\end{equation}

We acknowledge that for all $u \in L^{2}_{\Phi_0}(\C,\mathfrak{m})$, $e^{-\Phi_0/\hbar} \mathfrak{m} u \in L^2(\C)$. Therefore the new formula \eqref{eq.pseudolc} induces an operator $L^2(\C) \to L^2(\C)$ with kernel given by
\[k_\hbar(z,w) =  \frac{1}{2\pi \hbar}\left( p_{\Phi_0}(\tfrac{z+w}{2},\zeta(z,w))e^{-\frac{c}{\hbar}\frac{|z-w|^2}{\langle z-w \rangle}} \right) \mathfrak{m}^{-1}(w).\]
Using assumption \ref{eq:order} $\mathfrak{m}$ and the fact that $p \in S(U,\mathfrak{m})$ ensures us that, choosing $C$ small enough in \eqref{eq.pseudolc},
\[\exists C_0>0,~ \exists N \in \N, ~ \forall (z,w) \in \C^2, \quad |p_{\Phi_0}(\tfrac{z+w}{2},\zeta(z,w))| \leq C\mathfrak{m}(\tfrac{z+w}{2}) \leq C_0 \langle z-w \rangle^N \mathfrak{m}(w). \]
This implies the bound, for some other $C>0$
\[\forall (z,w) \in \C^2, \quad |k_\hbar(z,w)| \leq \frac{C}{2\pi \hbar}\langle z-w \rangle^{N} e^{-\frac{c}{\hbar}\frac{|z-w|^2}{\langle z-w \rangle}},\]
and we deduce that 
\[ C_\hbar := \sup_{z} \iint_{w \in \C} |k_\hbar(z,w)|\mathrm{d}L(w) = \underset{\hbar \to 0^+}{\mathcal{O}(1)} \hbox{ and } \widetilde{C}_\hbar := \sup_{w} \iint_{z \in \C} |k_\hbar(z,w)|\mathrm{d}L(z) = \underset{\hbar \to 0^+}{\mathcal{O}(1)}. \]
Schur's inequality implies that the formula \eqref{eq.pseudolc} defines a bounded operator $L^2_{\Phi_0}(\C,\mathfrak{m}) \to L^2_{\Phi_0}(\C)$. Let us denote it momentarily $\widetilde{P}_{\Phi_0}$ and notice that $P_{\Phi_0}, \widetilde{P}_{\Phi_0}$ are both equal on $e^{\Phi_0/\hbar} \mathscr{S}(\C)$ which is dense in $H_{\Phi_0}(\C,\mathfrak{m})$ for the $L^2_{\Phi_0}(\C,\mathfrak{m})$ norm. We therefore deduce that $\widetilde{P}_{\Phi_0}$ and $P_{\Phi_0}$ coincide on $H_{\Phi_0}(\C,\mathfrak{m})$. 
\end{proof}
Lemma \ref{linf} suggests that the tail of the integral \eqref{eq:pseudogc} defining $P_{\Phi_0}$ is ``exponentially small''. In the following Lemma, we will make the most of Schur's inequalities to prove this fact. Taking\footnote{Even though there is a conflict of notation with the $\delta>0$ defining $(E_1-\delta,E_2+\delta)$, it will be always clear to which $\delta$ we refer.} $\delta >0$, we can truncate the contour and consider
\begin{equation}
\label{eq:trunc}
\Lambda_{c,\delta}(z) = \left \{ (w,\zeta) ~ \left | ~ w  \in \C, ~ \zeta = \tfrac{2}{i}\partial_z \Phi_0\left( \frac{z+w}{2} \right) + ic \frac{\overline{z-w}}{\langle z-w \rangle} \hbox{ and } |z-w| < \delta \right \} \right. ,
\end{equation}
then the following lemma proves that the operator
\begin{equation}
\label{eq:approx}
\forall u \in H_{\Phi_0}(\C,\mathfrak{m}), \quad P_{c,\delta} u := \frac{1}{2\pi \hbar}\iint_{\Lambda_{c,\delta}(z)}  p_{\Phi_0} \left( \frac{w+z}{2}, \zeta \right) e^{i\zeta(z-w)/\h} u_\h(w)\mathrm{d}w \wedge \mathrm{d}\zeta
\end{equation}
approximates $P_{\Phi_0}$ with exponential sharpness.

%

\begin{lemma}
\label{lem:rem}
$P_{c,\delta}$ is a bounded linear operator $H_{\Phi_0}(\C,\mathfrak{m}) \longrightarrow L^2_{\Phi_0}(\C)$ satisfying
\begin{equation}
\forall u \in H_{\Phi_0}(\C,\mathfrak{m}), \quad   \lVert (P_{\Phi_0} - P_{c,\delta})u \rVert_{L^2_{\Phi_0}(\C)} \leq \mathcal{O}(e^{-\frac{\mathscr{C}}{\hbar}})\| u \|_{L^2_{\Phi_0}(\C,\mathfrak{m})}.
\end{equation}
\end{lemma}

In this article, we will use the indicator function of a set $A \subset \C$ defined by $\mathds{1}_A(z) := \left \{ \begin{array}{ccc}
1 \hbox{ if } z \in A \\
0 \hbox{ if } z \notin A
\end{array} \right.$,
and in particular we will also denote $\mathds{1}_{\complement \mathbb{B}_\delta(z)}(w)$ the indicator function associated with the complement of the ball $\{ w ~ | ~ |z-w| < \delta \}$.

\begin{proof}
Denoting $R = P_{\Phi_0} - P_{c,\delta}$ and using the estimate \eqref{eq:order}, for all $u \in H_{\Phi_0}(\C,\mathfrak{m})$,
\begin{align*}
 |Ru(z)| \leq \frac{C}{2\pi  \hbar} e^{\tfrac{\Phi_0(z)}{\hbar}}\iint_{w \in \C }\mathds{1}_{\complement\mathbb{B}_\delta(z)}(w) \langle z-w \rangle^N e^{-\frac{c}{\hbar}\frac{|z-w|^2}{\langle z-w \rangle}} |\mathfrak{m}(w)e^{-\Phi_0(w)/\hbar}u(w)|dw \wedge d\overline{w}.
\end{align*}
Similarly to the proof of the previous theorem, the approach with Schur's inequality allows us to conclude. In particular, we obtain that $P_{\Phi_0}-P_{c,\delta}$ is a bounded linear operator $H_{\Phi_0}(\C,\mathfrak{m}) \longrightarrow L^2_{\Phi_0}(\C)$ and so is $P_{c,\delta}$.
\end{proof}

Let us finish with a technical lemma that will be useful in the next sections and the next chapters. The idea is that the action of pseudo-differential operators on a set only depend on a small neighbourhood of this set.

\begin{lemma}[Pseudolocality of pseudo-differential operators]
\label{lem:loca}
Let us consider two sets $\Omega, U$ such that $\Omega$ is open and $\mathrm{dist}(\overline{U}, \partial \Omega)>0$ then there exists $\mathscr{C}, \varrho >0$ and a function $\mathcal{O}(e^{-\frac{\mathscr{C}}{\hbar}})$ such that 
\begin{equation}
\forall u \in L^2_{\Phi_0}(\C,\mathfrak{m}), \quad \| P_{\Phi_0}u \|_{L^2_{\Phi_0}(U,\mathfrak{m})} \leq \varrho|\|u\|_{L^2_{\Phi_0}(\Omega,\mathfrak{m})}+\mathcal{O}(e^{-\frac{\mathscr{C}}{\hbar}}) \|u\|_{L^2_{\Phi_0}(\C,\mathfrak{m})}.
\end{equation}
\end{lemma}
\begin{proof}
We have $P_{\Phi_0} = P_{c,\delta} + R$ and taking $\delta >0$ small enough in \eqref{eq:trunc} we make sure that the $w$-projection of
$\Lambda_{c,\delta}(z)$ is included in $\Omega
$ for all $z \in \overline{U}$. \\

We have $\| P_{c,\delta} u \|_{L^2_{\Phi_0}(U)} = \| \mathds{1}_{U} P_{c,\delta}\mathds{1}_{\Omega} u \|_{L^2_{\Phi_0}(\C)}$ and the integral kernel of   

$e^{-\Phi_0/\hbar}\mathds{1}_{U} P_{c,\delta}e^{-\Phi_0/\hbar}\mathds{1}_{\Omega}$ (seen as a kernel operator of $L^2(\C)$) is
\begin{equation}
k_\hbar(z,w) = \frac{1}{2\pi \hbar}\mathds{1}_{U}(z) \left( p_{\Phi_0}(\frac{z+w}{2},\zeta(z,w))e^{-\frac{C}{\hbar}\frac{|z-w|^2}{\langle z-w \rangle}} \right) \mathds{1}(|z-w|<\delta) \mathds{1}_{\Omega}(w)
\end{equation}
with $\zeta(z,w)$ given by \eqref{eq.contour}.
We notice that
\begin{align*}
C_\hbar := \sup_{z} \iint_{w \in \C} |k_\hbar(z,w)|\mathfrak{m}^{-1}(w)\mathrm{d}L(w) = \underset{\hbar \to 0^+}{\mathcal{O}(1)},
\end{align*}
and similarly for the integral in $w$ (denoting it $\widetilde{C}_\hbar$) so that there exists $C(\delta) >0$ such that $(C_\hbar \widetilde{C}_\hbar)^{\frac12} < C(\delta)$ for all $\hbar\in (0,\hbar_0)$. Schur's inequality for integral kernels yields
\begin{equation}
\| \mathds{1}_{U} P_{c,\delta} u\|_{L^2_{\Phi_0}(\C)} \leq (C_\hbar \widetilde{C}_\hbar)^{\frac12} \| \mathds{1}_{\Omega} u \|_{L^2_{\Phi_0}(\C,\mathfrak{m})} \leq C(\delta) \| u \|_{L^2_{\Phi_0}(\Omega,\mathfrak{m})}.
\end{equation}
Lemma \ref{lem:rem} yields  $\mathscr{C} >0$ such that
\[ \|P_{\Phi_0} u \|_{L^2_{\Phi_0}(U)} \leq C(\delta) \| u \|_{L^2_{\Phi_0}(\Omega,\mathfrak{m})} + \mathcal{O}(e^{-\frac{\mathscr{C}}{\hbar}}) \|u\|_{L^2_{\Phi_0}(\C,\mathfrak{m})}. \]
The conclusion follows.
\end{proof}

Using the FBI transform, we have taken advantage of the holomorphy of the symbol $p$ to reduce, modulo an exponentially small remainder, the definition of pseudo-differential operators to a local formula. This is key in order to prove that the WKB quasimodes we build in Section \ref{sec:WKB} are sharp. In the next section we provide the construction of the phase function as well as the geometric results behind the Maslov index.

\section{Analysis of the geometry} 
In this subsection, we introduce the geometric results that we will need in the following section. 
\label{sec:chap3_geomana}
\subsection{Action-Angle Lemma in the $L^2(\R)$ setting}
Similarly to the previous chapter, the energy sets $E_\lambda$ we are interested in are within the scope of the classification of $1$ dimensional compact connected submanifolds embedded in $\R^2$. Since we ask $dp \neq 0$ along these submanifolds it is standard, using curvilinear coordinates, that they are diffeomorphic to the $1$ dimensional torus $\T = \R/\Z$. In fact we even have the action-angle theorem illustrating that the geometry of this situation is not far from that of the excited harmonic oscillator and that we can expect similar results in the semiclassical regime. To define the action $A(\lambda)$, we parametrize $E_{\lambda}$ by a simple loop in $\R^2$ oriented by the Hamiltonian flow,
\begin{equation}
\forall \lambda \in (E_1-\delta,E_2+\delta), \quad A(\lambda) = \int_{E_{\lambda}} \xi dx.
\end{equation}
The usual results of Hamiltonian mechanics yield that, in that setting, the Hamiltonian flow is periodic of period $T(\lambda)$ given by
\begin{equation}
 T(\lambda) = A'(\lambda).
\end{equation}

\begin{figure}
\begin{tikzpicture}[scale = 1.2]
\draw[->] (5.95,-1.15) -- (5.85,1.5) node [above] {$\mathcal{A}$};
\draw[->] (5.95,-1.15) -- (7.1,-0.81);
\draw[->] (5.95,-1.15) -- (6.85,-0.65);
\draw[color=black,smooth] plot[domain=0:2*pi]
({cos(\x r)+6},0,{sin(\x r)});
\draw[color=black,smooth] plot[domain=0:2*pi]
({cos(\x r)+6},0.1,{sin(\x r)});
\draw[color=black,smooth] plot[domain=0:2*pi]
({cos(\x r)+6},-0.1,{sin(\x r)});
\draw[color=black,smooth] plot[domain=0:2*pi]
({cos(\x r)+6},0.2,{sin(\x r)});
\draw[color=black, opacity = 0.3,smooth] plot[domain=0:2*pi]
({cos(\x r)+6},1,{sin(\x r)});
\draw[color=black, opacity = 0.3,smooth] plot[domain=0:pi]
({cos((\x-0.2) r)+6},-1,{sin((\x-0.2) r)});
\draw[dashed, color=black, opacity = 0.3,smooth] plot[domain=0:pi]
({cos((-\x-0.2) r)+6},-1,{sin((-\x-0.2) r)});
\draw[color=gray] (4.93,0.85) -- (4.93,-1.13);
\draw[color=gray] (7.07,1.15) -- (7.07,-0.85);
\draw[<-] (7,-0.1) -- (8,0) node [right] {Action $A(E)$};
\draw[color=black] plot[domain=-0.65:0]
({0.5*cos((\x-0.2) r)+6},-1,{0.5*sin((\x-0.2) r)});
\draw (6.5,-1) node[below] {$\theta$};
\draw (5.5,-1.3) node[below, gray] {$\mathbb{S}^1$};
\draw[black] plot [smooth cycle, tension = 0.9] coordinates {(0,0) (1.5,0.5) (2,0) (1,-1)};
\draw[black] plot [smooth cycle, tension = 0.9] coordinates {(0.1,0) (1.4,0.4) (1.9,0) (1,-0.9)};
\draw[black] plot [smooth cycle, tension = 0.9] coordinates {(0.2,0) (1.3,0.3) (1.8,0) (1,-0.8)};
\draw[black, dotted] plot [smooth cycle, tension = 0.9] coordinates {(0.25,0) (1.2,0.25) (1.7,-0.05) (0.95,-0.75)};
\draw[black, dotted] plot [smooth cycle, tension = 0.9] coordinates {(0.35,0) (1.15,0.2) (1.65,-0.05) (0.95,-0.65)};
\draw[black, dotted] plot [smooth cycle, tension = 0.9] coordinates {(-0.1,0) (1.6,0.6) (2.1,0) (1,-1.1)};
\draw[black, dotted] plot [smooth cycle, tension = 0.9] coordinates {(-0.2,0) (1.7,0.7) (2.2,0) (1,-1.2)};
\draw[->] (1,-0.2) -- (2.5,-0.2);
\draw[->] (1,-0.2) -- (1,1);
\draw (1,1) node[above] {$\xi$};
\draw (2.5,-0.2) node[right] {$x$};
\draw[<-] (1,-0.8) -- (1.5,-1.25) node[right] {Energy $E$};
\draw[<->, double] (3,0) -- (4,0);
\end{tikzpicture}
\caption{Illustration of the action-angle change of variables.}
\label{fig:0}
\end{figure}
In the following lemma, we provide $\R^2$ with its natural symplectic structure $\mathrm{d}\xi \wedge \mathrm{d}x$ as well as $\mathbb{T}\times \R$ with $\mathrm{d}A \wedge \mathrm{d}\theta$ (letting $A$ be the second coordinate). 
\begin{lemma}[Action-Angle]
\label{lem:aa}
There exists a map $\theta : p^{-1}((E_1-\delta,E_2+\delta)) \cap \R^2  \longrightarrow \T$ such that the map defined as follows
\begin{equation}
\label{eq:hamsymp}
 \begin{array}{ccc}
p^{-1}((E_1-\delta,E_2+\delta)) \cap \R^2 &  \longrightarrow & \T \times \mathrm{Neigh}(A((E_1,E_2),\R)) \\
(x,\xi) & \longmapsto & (\theta,A(p(x,\xi))),
\end{array} 
\end{equation}
is a smooth symplectomorphism. It satisfies $\nu_*(\mathcal{X}_p) = \frac{1}{T(\lambda)}\frac{\partial}{\partial_\theta}$ where $\mathcal{X}_p$ is the Hamiltonian vector field given by
\begin{equation}
\mathcal{X}_p = \partial_\xi p \frac{\partial}{\partial x} -\partial_x p \frac{\partial}{\partial \xi}.
\end{equation}
\end{lemma}

\begin{remark} \label{rk:cxaa} In our setting, $p$ being holomorphic it is even possible to build a holomorphic symplectomorphism, for instance see Ito's article \cite{ito}. In the case of one degree of freedom one needs to show that the action $A$ is holomorphic (see Lemma \eqref{lem:hol}) and then find an holomorphic solution $\theta(x,\xi)$ of the Poisson bracket equation $\{\theta,A \circ p \} = -1$.
\end{remark}

\subsection{Holomorphy of the action and eikonal equations}
In what follows, we prove Lemma \ref{chgchap2} which is the main preliminary technicality allowing us to locally solve the eikonal equation 
\begin{equation}
\label{eq:eik3}
p_{\Phi_0}(z,\varphi_\lambda'(z)) = \lambda.
\end{equation}
 and prove that the action $A$ is holomorphic.
 
In the Bargmann setting, it is noteworthy that the natural projection
\begin{equation}
\label{eq:proj}
\pi_z : \begin{cases} \begin{array}{ccc}
(z,\zeta) & \longmapsto & z \\
\Lambda_{\Phi_0} & \longrightarrow & \C
\end{array} \end{cases},
\end{equation}
is an isomorphism. In that regard, we will often identify the sets 
\begin{equation}
\mathscr{E}_{\lambda} := \kappa_0(E_\lambda) = \{ (z,\zeta) \in \Lambda_{\Phi_0} ~ | ~ p_{\Phi_0}(z,\zeta) = \lambda \},
\end{equation} with the corresponding real curve in $\C$ (which is in particular maximally totally real).
 
\begin{lemma}
\label{chgchap2}
Let $\lambda_1 \in (E_1 - \delta,E_2 + \delta)$, we can chose $\mathrm{Neigh}(\mathscr{E}_{\lambda_1},\C^2)$, a small enough neighbourhood of $\mathscr{E}_{\lambda_1}$ in $\C^2$ on which the map $
\varrho : (z,\zeta) \longmapsto (z,p_{\Phi_0}(z,\zeta))$ is a biholomorphism onto its image.
\end{lemma}
\begin{proof}
The Jacobian of $\rho$, $\mathrm{Jac}(\rho)= \partial_{\zeta} p_{\Phi_0}$, is nonzero along $\mathscr{E}_{\lambda_1}$ because we assume $\mathrm{d}p \neq 0$ along $\{ (x,\xi) \in \R^2 ~ | ~ p(x,\xi) = \lambda_1 \}$. The local inverse function theorem shows that $\rho$ is locally a biholomorphism near any point of $\mathscr{E}_{\lambda_1}$. We can cover a small neighbourhood of $\mathscr{E}_{\lambda_1}$ with $m$ sets of the form
\[ \mathscr{U}_j := \{ (z,\zeta) \in \C^2 ~ | ~ |z-z_j| < \epsilon_j ~ \& ~ | \zeta - \tfrac{2}{i}\partial_z \Phi_0(z)| < \epsilon_j \}, ~~ j \in [\![1,m]\!], \]
where $(z_j,\tfrac{2}{i}\partial_z \Phi_0(z))$ is some point of $\mathscr{E}_{\lambda_0}$ and $\epsilon_j >0$ is small enough. We can further assume that on each $\mathscr{U}_j$ we can apply the local inverse function theorem providing a unique inverse $\zeta_j(\lambda,z)$ satisfying
\[ \forall (z,\zeta) \in \mathscr{E}_{\lambda_1} \cap \mathscr{U}_j, \quad \zeta_j(\lambda_1,z) := \frac{2}{i}\partial_z \Phi_0(z). \]
To conclude it suffices to notice that $\rho$ is injective in the union of $\mathscr{U}_j$. Suppose there exists $(z_1,\zeta_1),(z_2,\zeta_2) \in \bigcup_{j=0}^m \mathscr{U}_j$ such that $z_1 = z_2$ and $p_{\Phi_0}(z_1,\zeta_1) = p_{\Phi_0}(z_2,\zeta_2)$. Up to re-indexing, we can suppose that \\
---either $(z_1,\zeta_1),(z_2,\zeta_2) \in \mathscr{U}_1$ in which case the local inverse function theorem yields the uniqueness, \\
--- or $(z_1,\zeta_1) \in \mathscr{U}_1$ and $(z_2,\zeta_2) \in \mathscr{U}_2$. We have $z_1 = z_2$ and setting $\epsilon = \min (\epsilon_1,\epsilon_2)$ we have 
\[ \{(z_1,\zeta) ~ | ~ |\zeta - \tfrac{2}{i} \partial_z \Phi_0(z) | < \epsilon \} \subset \mathscr{U}_1 \cap \mathscr{U}_2 \]
and this proves that one of the two points $(z_1,\zeta_1)$, $(z_2,\zeta_2)$ belongs to $\mathscr{U}_1 \cap \mathscr{U}_2$. Therefore the uniqueness given by the local inverse function theorem yields $(z_1,\zeta_1) = (z_2,\zeta_2)$ and this seals the injectivity of $\rho$.
\end{proof}
The relevance of such a lemma appears when acknowledging that it provides a biholomorphism, for all $\lambda \in \C$ near $\lambda_1$, 
\begin{equation}
 \begin{array}{ccc}
z & \longmapsto & (z, \zeta(\lambda,z)) \\
\mathrm{Neigh}(\pi_z \mathscr{E}_{\lambda_1},\C) & \longrightarrow & \{ p_{\Phi_0} = \lambda \} \cap \mathrm{Neigh}(\mathscr{E}_{\lambda_1},\C^2).
\end{array}.
\end{equation}
This allows to solve easily the eikonal equation \eqref{eq:eik3}.

\begin{lemma}[Local eikonal resolution]
\label{lem:eko}
Let $\lambda_1 \in (E_1-\delta,E_2+\delta)$ and $\omega \in \mathscr{E}_{\lambda_1}$. The smooth solutions of the eikonal equation \eqref{eq:eik3} near $\omega$ associated with $\lambda \in \mathrm{Neigh}(\lambda_1,\C)$ and satisfying $\varphi_\lambda'(\omega) = \zeta(\lambda,\omega) \in \mathrm{Neigh}(\tfrac{2}{i}\partial_z \Phi_0(\omega),\C)$ are holomorphic and given by,
\begin{equation}
\label{eq:eko}
\forall z \in \mathrm{Neigh}(\omega,\C), \quad \varphi_\lambda(z) := \int_\omega^z \zeta(\lambda,w)dw + \mathrm{constant}.
\end{equation}
\end{lemma}
\begin{proof}
 Lemma \ref{chgchap2} ensures that there is a unique solution to the equation 
$p(z,\zeta) = \lambda$ near $(\omega,\tfrac{2}{i}\partial_z \Phi_0(\omega))$ given by $\zeta = \zeta(\lambda,z)$. Since $p_{\Phi_0}(z,\varphi_\lambda'(z)) = \lambda$ and $\varphi_\lambda(\omega) = \zeta(\lambda,\omega)$ we infer that for all $z$ near $\omega$, $\varphi_\lambda'(z) = \zeta(\lambda,z)$. It follows that $\varphi_\lambda$ is holomorphic and \eqref{eq:eko} is obtained by integration.
\end{proof}

Another useful consequence of Lemma \ref{chgchap2} is that the action $A(\lambda)$  is in fact holomorphic in a small complex neighbourhood of $(E_1,E_2)$. In the following lemma, the orientation on $\mathscr{E}_\lambda$ will be given by the Hamiltonian flow in the sense that $\mathscr{E}_\lambda$ is identified with a simple loop in $\Lambda_{\Phi_0}$ given by the Hamiltonian flow of a point in $\mathscr{E}_\lambda$.

\begin{lemma}[Holomorphic Quantum Action]
\label{lem:hol}
The action $A$ extends by holomorphy to a complex neighbourhood of $[E_1,E_2]$ and we will denote by $\mathcal{A}$ this extension. 
\end{lemma}
\begin{proof}
First, let $\lambda \in (E_1-\delta,E_2+\delta)$, we have on $\Lambda_{\Phi_0}$, 
\[ \kappa_0^*(\zeta dz) = \frac{1}{2}(\xi-ix) d(x-i\xi) = \frac12(\xi dx - x \mathrm{d} \xi) - i \kappa_0^*(d \Phi_0) = \xi \mathrm{d} x - \frac12 \mathrm{d}(x\xi) - i \kappa_0^*(d \Phi_0),\] so that
\begin{equation}
\label{eq:actionchap3}
A(\lambda) = \int_{E_{\lambda}} \xi dx = \int_{E_{\lambda}} \xi dx - \frac12 \mathrm{d}(x\xi) = \int_{\mathscr{E}_{\lambda}} (\zeta dz - id\Phi_0) = \int_{\mathscr{E}_{\lambda}} \zeta dz.
\end{equation}
For all $\lambda$ in a small real neighbourhood of $\lambda_1$, by the change of contour $\pi_z (\mathscr{E}_{\lambda}) \to \pi_z(\mathscr{E}_{\lambda_1})$ and the holomorphy of $(\lambda,z) \longmapsto \zeta(\lambda,z)$ we have
\begin{equation}
A(\lambda) =  \int_{\mathscr{E}_{\lambda}} \zeta dz = \int_{\pi_z (\mathscr{E}_{\lambda})} \zeta(\lambda,z)dz = \int_{\pi_z(\mathscr{E}_{\lambda_1})} \zeta(\lambda,z)dz.
\end{equation}
We notice that $\mathcal{A}(\lambda)=\int_{\pi_z(\mathscr{E}_{\lambda_1})} \zeta(\lambda,z)dz$ defines a holomorphic map for $\lambda$ near $\lambda_1$. Since this is true for any $\lambda_1 \in (E_1-\delta,E_2+\delta)$ and by uniqueness of $A(\lambda)$ for $\lambda \in (E_1-\delta,E_2+\delta)$, the conclusion stems from standard patching results for holomorphic functions.
\end{proof}

\begin{remark}
\label{rk:action}
Letting $\lambda \in (E_1-\delta,E_2+\delta)$ and using the ``real expression'' giving $A(\lambda)$, we notice that it corresponds to the area in $\R^2$ enclosed by $E_\lambda$. This proves that the action does not vanish thus the phases $\varphi_\lambda$ which are primitives of $\zeta(\lambda,z)$ suffer from multivaluedness issues on a neighbourhood of $\mathscr{E}_\lambda$. In other words $\varphi_\lambda$ is only well defined on the universal cover (or the corresponding Riemann surface).
\end{remark}

\subsection{Index of the energy sets}
Now we finish this section with the geometric result underpinning the Maslov correction that we will determine in Section \ref{sec.maslov}. \\

In the usual functional setting $L^2(\R)$, the energy sets $E_\lambda$ cannot be projected on the variable $x$ and this is an obstacle to building WKB quasimodes of the form $a_\hbar(x) e^{\frac{i}{\hbar}\phi(x)}$. Indeed at the critical points $x$ where the projection 
\[ E_\lambda \ni (x,\xi) \longmapsto x \]
is not locally an isomorphism, the phase $\phi$ is not well defined. A way of getting around this issue is by looking at generalised WKB quasimodes defined either with special functions or as the inverse Fourier transform of functions of WKB type.  The Maslov correction then appears as the condition on the eigenvalues to be able to glue the ``usual'' WKB quasimodes with the generalized ones defined at the ``turning points''.

\label{sec:index}
In the Bargmann's setting $H_{\Phi_0}(\C)$, \eqref{eq:proj} being an isomorphism, the energy sets $\mathscr{E}_\lambda$ are projectable on the variable $z$, thus the lack of projectability is not of concern. As we will see in Section \ref{sec.maslov}, the Maslov condition will still stem from a gluing condition. In particular it is due to the particular form of the principal symbol $\partial_\zeta p_{\Phi_0}(z,\zeta(\lambda,z))^{-\frac12}$ of the WKB expansion that we will find in Section \ref{sec:WKB}.  We will see in Section \ref{sec.maslov} that it amounts to the following result (Lemma \ref{lem:rac}) on the index of the energy sets in $\C$.  \\

We let $\lambda \in (E_1-\delta,E_2+\delta)$ and define the loop $\gamma_\lambda : [0,T(\lambda)] \ni t \longmapsto z(t)$ with the Hamiltonian flow $(z(t),\zeta(t))$ of a point $(z,\zeta) \in \mathscr{E}_\lambda$. We recall that with $(x(t),\xi(t)) = \kappa_0^{-1}(z(t),\zeta(t))$
\[ \frac{\mathrm{d}z(t)}{\mathrm{d}t} = \frac{\mathrm{d}x(t)}{\mathrm{d}t}-i\frac{\mathrm{d}\xi(t)}{\mathrm{d}t}= \partial_\xi p(x(t),\xi(t)) +i \partial_x p(x(t),\xi(t))  = \partial_\zeta p_{\Phi_0}(z(t),\zeta(t)). \]
By Assumption \ref{enum:2}, $\nabla p$ does not vanish on $E_{\lambda}$ thus $\gamma_\lambda'(t) = \partial_\zeta p_{\Phi_0}(z(t),\zeta(t)) \neq 0$. The goal is to prove the following lemma that will be useful for computing the subprincipal term of the Bohr-Sommerfeld expansion in Section \ref{sec.maslov}. We are mainly interested in computing the index of the derivative of $\gamma_\lambda$. The idea is that the Gauss-Bonnet theorem allows to link it with that of $\gamma_\lambda$. In the following lemma, a squareroot of a smooth loop $\sigma : [0,T] \to \C^*$, $T>0$, is a smooth path (which is not necessarily a loop) $\gamma : [0,T] \to \C^*$ satisfying $\gamma(t)^2 = \sigma(t)$ for all $t \in [0,T]$. 

\begin{lemma}\label{lem:rac}
Let us consider the complex loop $\sigma : [0,T(\lambda)] \ni t \longmapsto \partial_\zeta p_{\Phi_0}(z(t),\zeta(t))$. It admits two squareroot paths $\sqrt{\sigma}, -\sqrt{\sigma}$ which are well defined on $[0,T(\lambda)]$  thanks to the path lifting property on $\C^*$ and each squareroot path corresponds to a choice of branch for the squareroot of $\sigma(0)$. To make it more concrete, let us say that $\sqrt{\sigma}(0)$ is set by the squareroot branch 
\[ \R_+^* \times [0,2\pi) \ni (r,\theta) \mapsto re^{i\theta} \mapsto \sqrt{r}e^{i\theta/2}. \]
 The squareroot path $\sqrt{\sigma}$ then satisfies 
\begin{equation}
\sqrt{\sigma}(0) = -\sqrt{\sigma}(T(\lambda)).
\end{equation}
\end{lemma}

In particular, this results prevents the existence of a squareroot of $\partial_\zeta p_{\Phi_0}(z,\zeta(\lambda,z))$ near the $z$-projection of $\mathscr{E}_{\lambda}$.

\begin{proof}
Let us first recall the definition of the index of a smooth closed curve $\gamma: [0,t_0] \longmapsto \C$ in the complex plane revolving around a point $a\in \C$. The universal cover of $\C\setminus \{a \}$ is $\C$ (thanks to the exponential function). By uniqueness of the lifted path, there exists a unique lifting $\theta : [0,t_0] \ni t \longmapsto \theta(t)\in \C$ of $\gamma$ such that $\gamma(t) =a + (\gamma(0)-a)e^{2i\pi\theta(t)}$ and $\theta(0) = 0$. The index of $\gamma$ around $a$ is then unambiguously given by $\mathrm{Ind}(\gamma) := \theta(t_0)-\theta(0) \in \Z$.

The index enjoys, in particular, the property of being \emph{constant on the homotopy classes of loops} as well as the following property that will be useful. In the case when \emph{$\gamma'$ does not vanish} on $[0,t_0]$, if we consider the \emph{tangential Gauss map} defined by $[0,t_0] \ni t \longmapsto \frac{\gamma'(t)}{|\gamma'(t)|} \in \mathbb{S}^1$, we have the following standard result as consequence of Gauss-Bonnet's Theorem. 

\begin{lemma}
The index of the smooth loop $\gamma$ (such that $\gamma'$ does not vanish) around $a$ is also the index of its tangential Gauss map around $0$.
\end{lemma}
\begin{proof}
Let us proceed to a first a priori simplification: we can suppose that $\gamma$ has constant velocity \emph{i.e.} for each $t \in [0,t_0]$, $|\gamma'(t)| = 1$. To see that we define $\beta(t) = \int_0^t |\gamma'(s)|ds$ and notice that it is a diffeomorphism $[0,t_0] \longmapsto [0,\beta(t_0)]$ since by assumption $\gamma'$ does not vanish. If we denote $\alpha$ its reciprocal we obtain that
\[ \frac{\mathrm{d}}{\mathrm{d}t}\gamma(\alpha(t)) = \alpha'(t) \gamma'(\alpha(t)) = \frac{1}{|\gamma'(\alpha(t))|} \gamma'(\alpha(t)) \]
is unitary. The parametrization $[0,\beta(t_0)] \ni t \longmapsto \gamma(\alpha(t))$ then satisfies the desired properties and is homotopic to $\gamma$.

In complex notation, the geodesic curvature at $t \in [0,t_0]$ of the curve $\gamma$ is then simply $k_g(t) = \mathrm{Im}(\overline{\gamma'}(t) \gamma''(t))$. We then assume that $\gamma$ is a simple loop around $a$. Thanks to Gauss-Bonnet's theorem
\[\mathrm{Im}\left(\int_{0}^{t_0}\overline{\gamma'}(t) \gamma''(t) \mathrm{d}t\right) = \int_{0}^{t_0}k_g(t) \mathrm{d}t = 2 \pi \mathrm{Ind}_a(\gamma).  \]
Notice then that the previous formula extends by concatenation of loops to the case where $\gamma$ is not a simple loop. The famous residue theorem for meromorphic functions allows to link the winding number with complex analysis
\[\mathrm{Ind}_0(\gamma') = \frac{1}{2 \pi}\mathrm{Im} \left(\int_{\gamma'} \frac{\mathrm{d} z}{z} \right) = \frac{1}{2 \pi} \mathrm{Im}\left(\int_{0}^{t_0}\overline{\gamma'}(t) \gamma''(t) \mathrm{d}t\right) = \mathrm{Ind}_a(\gamma).  \]

\end{proof}

To conclude, the loop $\gamma_\lambda : [0,1] \ni t \longmapsto z(t)$ is a simple, closed and regular loop whose index is $\pm 1$. By homotopy, the tangential Gauss map of $\gamma_\lambda$ has the same index as $\gamma_\lambda'$. Thanks to the previous lemma, the loop $\sigma:[0,T(\lambda)] \ni t \longmapsto \gamma_\lambda'(t)= \partial_\zeta p_{\Phi_0}(z(t),\zeta(t))$ is a loop in $\C$ of index $1$ or $-1$ around $0$.  It then suffices to recall the following natural result.

\begin{lemma}
\label{lem:sqrtmoins}
Let $\mathfrak{s}$ a loop in $\C^*$ around $0$, it admits two smooth squareroots paths that we denote arbitrarily $-\sqrt{\mathfrak{s}}, \sqrt{\mathfrak{s}}$, we have
\begin{enumerate}
\item $\sqrt{\mathfrak{s}}(T(\lambda)) = - \sqrt{\mathfrak{s}}(0)$ if $\mathrm{Ind}(\mathfrak{s}) \in 2\Z+1$, 
\item $\sqrt{\mathfrak{s}}(T(\lambda)) = \sqrt{\mathfrak{s}}(0)$ if $\mathrm{Ind}(\mathfrak{s}) \in 2\Z$. 
\end{enumerate}
\end{lemma}

The loop $[0,T(\lambda)] \ni t \longmapsto  \partial_\zeta p_{\Phi_0}(z(t),\zeta(t))$ having index $\pm 1$ around $0$, we can then apply Lemma \ref{lem:sqrtmoins} and finish the proof of Lemma \ref{lem:rac}.
\end{proof}

\begin{remark} In \cite{Arnold1978} V.Arnold developped the framework allowing to associate an index: \emph{Maslov's index} with a \emph{loop of Lagrangian spaces} in the Lagrangian Grassmanian $\Lambda(n)$ of $\R^{2n}$. In our situation Lemma \ref{lem:rac} is linked to the computation of the Maslov's index of the path $[0,T_0] \ni t \longmapsto \R \gamma'(t)\in \Lambda(1)$ and the relation between the indices comes from the diffeomorphism $\Lambda(1) \simeq \R \mathrm{\textbf{P}}^1 \simeq \mathbb{S}^1$. 
\end{remark}

\section{Construction of global WKB quasimodes}
\textit{In this section we start recalling a classical $\mathcal{O}(h^2)$ result and then prove, extending the ideas of Fujie-Zerzeri \cite{fujie}, the existence of WKB quasimodes which are exponentially sharp.}
\label{sec:WKB}
\subsection{Normal form in the smooth setting}

In this chapter, we will need a weak $\mathcal{O}(\hbar^2)$ version of a result that is stated with precision $\mathcal{O}(\hbar^\infty)$ in V\~u Ng\d{o}c's \cite[Proposition 5.1.3]{pano} or \cite[Section 5]{San2000}. The idea of the following lemma is that the eigenvalues of $P_{\Phi_0}$ are contained in the union of the small intervals $[\mathcal{A}^{-1}(2\pi \hbar (k+\frac12)) -\mathscr{C}\hbar^2, \mathcal{A}^{-1}(2\pi \hbar (k+\frac12)) +\mathscr{C}\hbar^2]$.

\begin{lemma}[Bohr-Sommerfeld $\mathcal{O}(\hbar^2)$]
\label{lem:bscinfinity}
There exists a constant $\mathscr{C}>0$ such that for $\hbar$ small enough
\begin{equation}
\label{eq:action}
\forall \lambda_\hbar \in \sigma_\hbar(P)  \cap \mathrm{Neigh}(\lambda_0,\R), ~ \exists ! k \in \Z, ~~ |\mathcal{A}(\lambda_\hbar) -2\pi \hbar (k+\tfrac12)| \leq \mathscr{C}\hbar^2.
\end{equation}
Moreover the eigenvalues of $P$ are of multiplicity $1$ and there is at most one eigenvalue satisfying \eqref{eq:action}.
\end{lemma}

In the lemma, we use the notation $ \mathrm{Neigh}(\lambda_0,\R)$ to denote a small $\hbar$-independent neighbourhood of $\lambda_0$ in $\R$.

\begin{proof}
We simply provide the philosophy of the proof. Considering $\lambda \in (E_1-\delta,E_2+\delta)$ the standard $\mathcal{O}(\hbar^2)$ analysis with Fourier Integral Operators guarantees the existence and uniqueness of microlocal solutions (defined with Fourier Integral Operators) of the equation $(P-\lambda)u = 0$ near any point of the energy curve. Since the eigenfunctions are microlocalized near $E_\lambda$ the existence of an eigenfunction of $P$ is equivalent to the existence of such semiglobal microlocal solution. This  semiglobal microlocal solution exists if and only if it is possible to patch the fully microlocal solutions along the energy curve. The condition on the transition between these microlocal solutions stemming from \v Cech cohomology yields that, in that case, $\lambda$ satisfies \eqref{eq:action}. 
\end{proof}

We now have a good idea of the shape of the spectrum. It remains to prove the reciprocal of the previous lemma, namely that we can build quasimodes giving an exponentially sharp approximation of the eigenvalues. 

\subsection{Microlocal WKB Expansions}
\label{sec:micro}

\subsubsection{Construction of analytic WKB expansions}
We start with finding microlocal WKB expansions. We let $z_0 \in \mathscr{E}_{\lambda_0}$ and denote $\phi(\lambda,z)$, for $\lambda$ near $\lambda_0$, the phase given by
\begin{equation}
\phi(\lambda,z) = \int_{z_0}^z \zeta(\lambda,z)\mathrm{d}z, \hbox{ satisfying } p_{\Phi_0}(z,\partial_z \phi(\lambda,z)) = \lambda.
\end{equation}
Let us mention that we will sometimes use $\phi_\lambda$ to refer to $\phi(\lambda,\cdot)$ and $\phi_\lambda' = \partial_z \phi_\lambda$. We also define a pluriharmonic gauge $\Phi(\lambda,z) = - \mathrm{Im}\phi(\lambda,z)$. In order to build microlocal WKB quasimodes for $P_{c,\delta}-\lambda$, we will first build quasimodes $a(\hbar,\lambda,z)e^{\frac{i}{\hbar}\phi(\lambda,z)}$ in the germ spaces $H_{\Phi,(\lambda_0,z_0)}$ (see Appendix \ref{sec:pseudo1}).  The link with the operator $P_{\Phi_0}$ will then appear in Proposition \ref{prop:subhtoh}. The reason for which we start with looking for quasimodes in $H_{\Phi,(\lambda_0,z_0)}$ are the following: \\
--- it is easier to solve WKB equations $(P-\lambda)u \sim 0$ in spaces $H_{\Phi}$ with holomorphic gauges $\Phi$, \\
--- seeing $\lambda$ as a variable allows to link the microlocal solutions.

In Section \ref{sec:pseudo1} of the appendix, we recalled the construction of pseudodifferential operators on germ spaces. This allows to define a pseudodifferential operator
$\widetilde{P} : H_{\Phi,(\lambda_0,z_0)} \to H_{\Phi,(\lambda_0,z_0)}$
formally given by the integral
\begin{equation}
\label{def:ptildechap2}
\widetilde{P} u_\hbar = 
\frac{1}{2 \pi \h} \iint_{\mathscr{G}_{c,\delta}(\lambda,z)} e^{i \zeta(z-w)/ \h} p_{\Phi_0}(\tfrac{z+w}{2},\zeta)u_\h(w) \mathrm{d}\zeta \wedge \mathrm{d}w,
\end{equation}
where 
\[ \mathscr{G}_{c,\delta}(\lambda,z) := \left \{ (w,\zeta) ~ | ~ |z-w| \leq \delta, ~ \zeta = \tfrac{2}{i}\partial_z \Phi(\lambda,\tfrac{z+w}{2}) + i c\overline{z-w} \right \}, \]
is a smooth family of good contours for $(w,\zeta) \longmapsto - \mathrm{Im}(\zeta(z-w)) + \Phi(\lambda,w)$. 

\begin{proposition}[Microlocal WKB expansions]
\label{prop:chap3_qmodes}
There exists $a(\hbar,\lambda,z)$ a classical analytic symbol in $(\lambda,z)$ defined near $(\lambda_0,z_0)$ such that the following holds. \\
(1) There exists $\mathscr{C} >0$ such that, in a small neighbourood of $(\lambda_0,z_0)$,
\begin{equation}
\label{eq:bkw}
| e^{-\frac{i}{\hbar} \phi}(\widetilde{P}-\lambda)(e^{\frac{i}{\hbar} \phi}a(\hbar,\lambda,\cdot)) | \leq \mathcal{O}(e^{-\frac{\mathscr{C}}{\hbar}}),
\end{equation}
where $\widetilde{P}$ is defined in \eqref{def:ptildechap2} for $\delta$ small enough.\\
(2) The principal symbol $a_0(\lambda,z)$ of $a(\hbar,\lambda,z)$ is a squareroot of $w \mapsto \partial_\zeta p_{\Phi_0}(w,\zeta(\lambda,w))^{-1}$. \\
(3) If $b(\hbar,\lambda,z)$ is a classical analytic symbol such that the WKB expansion \\ $b(\hbar,\lambda,z)e^{\frac{i}{\hbar}\phi_\lambda(z)}$ satisfies equation \eqref{eq:bkw}, then there exists a classical analytic symbol $g(\hbar,\lambda)$  satisfying
\[ \forall \lambda \in \mathrm{Neigh}(\lambda_0,\C), ~ \forall z \in \mathrm{Neigh}(z_0,\C), \quad 
|b(\hbar,\lambda,z) -g(\hbar,\lambda) a(\hbar,\lambda,z)| \leq \mathcal{O}(e^{-\frac{\mathscr{C}}{\hbar}})\,. \]
\end{proposition}
\begin{proof}

$\textit{(1)}$  The main result for finding WKB expansions $a(\hbar,\cdot,\cdot) e^{\frac{i}{\hbar} \phi}$ in the analytic setting is Sjöstrand \cite[Theorem 2.8.1]{minicourse} that we recall here. The minimal ingredients for understanding the following result (among others the definitions of the space $H_{0,z_0}$) are stated in Appendix \ref{apx.1}. We have kept Sjöstrand notations and warn the reader that in the following lemma $(x,\xi)$ denotes a point in $\C^{2d}$ (with $d \in \N^*$) not a point in $\R^2$ as in the beginning of the article.  
\begin{theorem}[Theorem 2.8.1 \cite{minicourse}]
\label{JJ}
Let $f \in \mathrm{Hol}(\mathrm{Neigh}(x_0,\C^d))$ and 
$Q(x,\h D_x,\h) : H_{-\mathrm{Im}f,x_0} \to H_{-\mathrm{Im}f,x_0},$ be a complex pseudo-differential operator of order $0$ such that the leading symbol satisfies
\begin{equation}
q(x_0,\xi_0) = 0, ~ \partial_{\xi_d} q(x_0,\xi_0) \neq 0.
\end{equation}
Assume that $f$ solves the eikonal equation
\begin{equation}
\left \{ \begin{array}{ll}
q(x,f'(x)) = 0, \\
f'(0) = \xi_0.
\end{array} \right.
\end{equation}
Let $H$ be the hypersurface $\{ x_n = x_{0,d} \}$ and denote $x = (x',x_d) \in \C^d$. Let $v(x,\h),\underline{a}(x',\h)$ be classical analytic symbols of order $0$ defined near $x_0$ and $x_0'$ respectively. There exists a unique (modulo equivalence in $H_{0,x_0}$) classical analytic symbol $a(\h,\cdot)$ defined near $x_0$ such that 
\begin{equation}
e^{-i f/\h} Q e^{if/\h}a \sim \h v \hbox{ in } H_{0,x_0}, \quad a_{| H} = \underline{a}.
\end{equation}
\end{theorem}
We apply this theorem with $x = (\lambda,w) \in \mathrm{Neigh}((\lambda_0,z_0),\C^2)$, $Q = \widetilde{P}-\lambda$, $v = 0$, $f(\lambda,w) = \phi_\lambda(w)$ and $H = \{ w = z_0\}$ by asking $\underline{a}(\hbar,\lambda) = \partial_\zeta p_{\Phi_0}(z_0,\zeta(\lambda,z_0))^{-\frac12}$ (choosing a local squareroot) for all $\lambda$ near $\lambda_0$. This gives a classical analytic symbol denoted $a(\hbar,\lambda,w)$ defined near $z_0$ satisfying
\begin{equation}
\label{eq:symbanal}
\exists \mathscr{C} \in \R_+^*, ~ \forall \lambda \in \mathrm{Neigh}(\lambda_0,\C), \quad |e^{-\frac{i}{\hbar}\phi_\lambda}(\widetilde{P} - \lambda)e^{\frac{i}{\hbar}\phi_\lambda}a(\hbar,\lambda,\cdot)| \leq \mathcal{O}(e^{-\frac{\mathscr{C}}{\hbar}}).
\end{equation}

\begin{remark}
We will se that we choose $\underline{a}(\hbar,\lambda) = \partial_\zeta p_{\Phi_0}(z_0,\zeta(\lambda,z_0))^{-\frac12}$ in order to simplify the computation of the principal symbol of $a(\hbar,\lambda,z)$.
\end{remark}

$\textit{(2)}$ Let us compute the principal symbol $a_0$. For that, we define $P_{\phi} :=e^{-\frac{i}{\hbar}\phi}\widetilde{P}e^{\frac{i}{\hbar}\phi}$ and notice that for all $u \in \mathrm{Hol}(\mathrm{Neigh}((z_0,\lambda),\C^2)$,
\begin{align}
\label{eq:bs_Pphi}
P_{\phi} u(\lambda,z) = \frac{1}{2\pi \hbar}\iint_{\mathscr{G}_{c,\delta}(\lambda,z)} e^{\frac{i}{\hbar}\zeta(z-w)} p_{\Phi_0}(\tfrac{z+w}{2},\zeta)e^{\frac{i}{\hbar}(\phi(w)-\phi(z))} u(w)\mathrm{d}\zeta \wedge \mathrm{d}w.
\end{align}
Let us mention that, by Lemma \ref{lem:conjanaly}, $P_\phi : H_{0,(\lambda_0,z_0)} \to H_{0,(\lambda_0,z_0)}$ is a well defined complex pseudodifferential operator. Indeed, the function $\theta(\lambda,z,w) = \frac{\phi_\lambda(w)-\phi_\lambda(z)}{z-w}$ extended by $-\phi_\lambda'(z)$ at $w = z$, is holomorphic therefore $\Lambda_{\phi}(z) := \{ (w,\zeta +\theta(\lambda,z,w)) ~ | ~ (w,\zeta) \in \mathscr{G}_{c,\delta}(\lambda,z) \}$ defines a good contour for the phase $(w,\eta) \longmapsto -\mathrm{Im}(\eta \cdot(z-w))$. Thanks to the equivalence of good contours we can suppose, making only an error $\mathcal{O}(e^{-\frac{\mathscr{C}}{\hbar}})$, that the new contours are $\lambda$ independent and given by
\[\mathscr{G}^0_{c,\delta}(\lambda,z) := \{ (w,\zeta) ~ | ~ \zeta = ic\overline{z-w}, ~ |z-w| < \delta \}, \]
for a different $\delta>0$. Neglecting this exponentially small remainder, Formula \eqref{eq:bs_Pphi} becomes
\begin{align}
\label{eq:bs_Pphi2}
P_{\phi} u(z) = \frac{1}{2\pi \hbar}\iint_{\mathscr{G}^0_{C,\delta}(z)} e^{\frac{i}{\hbar}\zeta(z-w)} p_{\Phi_0}(\tfrac{z+w}{2},\zeta - \theta(\lambda,z,w))u(w) \mathrm{d}\zeta \wedge \mathrm{d}w.
\end{align}
Therefore the classical analytic symbol $a(\hbar,\lambda,z)$ satisfies \eqref{eq:bkw} if and only if it solves $(P_{\phi} -\lambda)a(\hbar,\lambda,\cdot) = 0$ in $H_{0,z_0}$. Formula \eqref{eq:bs_Pphi2} defines a complex pseudo-differential operator (with classical analytic symbol) on the space $H_{0,z_0}$. See for instance in \cite[Section 2.3]{RSV} for a comprehensive presentation of complex pseudo-differential operators and their properties. Now thanks to \cite[Section 3.2]{RSV} it is possible to find a classical analytic symbol $b_\hbar$ such that 
$
P_{\phi} = \mathrm{Op}_\hbar^w(b_\hbar)
$
where the formal expansion of $b_\hbar$ is given by 
\begin{equation}
\label{eq:oif}
b_\hbar(\lambda,z,\zeta) = \left.\exp\left[-\tfrac{i}{2\hbar}\hbar D_\zeta(\hbar D_z - \hbar D_w)\right ] p_{\Phi_0}(\tfrac{z+w}{2},\zeta-\theta(\lambda,z,w))\right|_{w = z}.
\end{equation}
Let us provide a quick formal explanation of the result to the reader in order to make it feel more natural. In order to shorten the notations, we will write $\phi_\lambda$ instead of $\phi$ and $\phi_\lambda'$ instead of $\partial_z \phi$. The first step is to write (at least formally) the infinite Taylor expansion of $p_{\Phi_0}(\frac{z+w}{2},\zeta-\theta(\lambda,z,w))$ centered at $\frac{z+w}{2}$. Namely one finds $p_1,\dots,p_k, \dots $ holomorphic such that
\begin{multline}
\label{eq:tayl}
p_{\Phi_0}(\tfrac{z+w}{2},\zeta-\theta(\lambda,z,w)) = p_{\Phi_0}(\tfrac{z+w}{2},\zeta+\phi_\lambda'(\tfrac{z+w}{2})) + (z-w)^2 p_1(\lambda,\tfrac{z+w}{2},\zeta) \\
+ (z-w)^4 p_2(\lambda,\tfrac{z+w}{2},\zeta)+ \cdots + (z-w)^{2k} p_k(\lambda,\tfrac{z+w}{2},\zeta) + \cdots
\end{multline}
The idea of the transformation \eqref{eq:oif} is justified by applying the formal integration by parts, with $\mathscr{G} := \{(w, i \overline{z-w}) ~ | ~ |z-w| \leq \epsilon \}$, for all $k \in \N$,
\begin{equation}
\begin{split}
\iint_{\mathscr{G}} e^{\frac{i}{\hbar}\zeta(z-w)}\frac{(z-w)^k}{k!}q(\lambda,\tfrac{z+w}{2},\zeta) u(w)\mathrm{d} \zeta \wedge \mathrm{d} w \\ = \frac{(i\hbar)^k}{k!} \iint_{\mathscr{G}}e^{\frac{i}{\hbar}\zeta(z-w)}\partial_\zeta^k q(\lambda,\tfrac{z+w}{2},\zeta) u(w)\mathrm{d} \zeta \wedge \mathrm{d} w,
\end{split}
\end{equation}
to the Taylor expansion \eqref{eq:tayl}. The proof provided in \cite[Section 3.2]{RSV} consists in proving that \eqref{eq:oif} is in fact a (``meta'') complex Fourier Integral Operator sending classical analytic symbols towards classical analytic symbols.

Finally, thanks to the (analytic) stationary phase, we find that in a small neighbourhood of $z_0$
\begin{equation}
\begin{aligned}
(P_{\phi}-\lambda )a_0  &= \mathrm{Op}_\hbar^w\bigl(p_{\Phi_0}(\tfrac{z+w}{2},\zeta +\phi_\lambda'(\tfrac{z+w}{2}))\bigr)a_0+ \mathcal{O}_{L^\infty}(\hbar^2) \\ 
& = (p_{\Phi_0}(z,\phi_\lambda'(z))-\lambda) a_0 + \mathrm{Op}_\hbar^w(\zeta \partial_\zeta p_{\Phi_0}(\tfrac{z+w}{2},\phi_\lambda'(z)))a_0 + \mathcal{O}_{L^\infty}(\hbar^2) \\
& = \frac{\hbar}{i}\left(\frac{1}{2}\partial_z (\partial_\zeta p_{\Phi_0}(z,\phi_\lambda'(z)))+ \partial_\zeta p_{\Phi_0}(z,\phi_\lambda'(z)) \partial_z \right)a_0+ \mathcal{O}_{L^\infty}(\hbar^2).
\end{aligned}
\end{equation}

This justifies that $a_0$ satisfies the transport equation
\begin{equation}
\label{eq:a0}
\left(\frac{1}{2}\partial_z (\partial_\zeta p_{\Phi_0}(z,\phi_\lambda'(z)))+ \partial_\zeta p_{\Phi_0}(z,\phi_\lambda'(z)) \partial_z \right)a_0 = 0.
\end{equation}
The way to solve this equation is to remember that $a_0$ should be linked to a half form (see \cite[Theorem 10.6]{Zworski} for a more comprehensive discussion) and hence that it could be fruitful to deduce an equation on $a_0^2$. Multiplying the previous equation with $a_0$ we simply get 
\begin{equation}
\partial_w \left( a_0^2 \partial_\zeta p_{\Phi_0}(w,\phi_\lambda'(w)) \right) = 0.
\end{equation}
This proves that $a_0^2(\lambda,w) = \partial_\zeta p_{\Phi_0}(w,\phi_\lambda'(w))^{-1}$ and if we set the initial condition, by choosing a local branch of the squareroot, $a_0^2(\lambda_0,z_0) = \partial_\zeta p_{\Phi_0}(z_0,\phi_\lambda'(z_0))^{-\frac12}$ we deduce that for all $(\lambda,w)$ near $(\lambda_0,z_0)$ 
\begin{equation}
\label{eq:pcpsymb}
a_0(\lambda,w) = [\partial_\zeta p_{\Phi_0}(w,\phi_\lambda'(w))]^{-\frac{1}{2}}.
\end{equation}

\textit{(3)} We now prove the last statement as follows. We consider
$u(\hbar,\lambda,\cdot)= \tfrac{a(\hbar,\lambda,z)}{a(\hbar,\lambda,z_0)}$ and notice that it is a classical analytic symbol as quotient of classical analytic symbols (see \ref{apx.1} in which we prove that the class of classical analytic symbols with non vanishing principal symbol is stable by multiplication and inverse). It satisfies in particular:

i) $(P_{\phi}-\lambda) u(\hbar,\lambda,z) = 0$ in $H_{0,(\lambda_0,z_0)}$,

ii) the formal analytic symbol $\sum_{n \geq 0} u_n(\lambda,z) \hbar^n$ of $u(\hbar,\lambda,z)$ satisfies $u_0(\lambda,z_0) = 1$ and $u_n(\lambda,z_0) = 0$ for $n \geq 1$.

It is then easy to check that both classical analytic symbols $b(\hbar,\lambda,z)$ and \\ $b(\hbar,\lambda,z_0) u(\hbar,\lambda,z)$ satisfy for some $\mathscr{C}>0$
\[ (P_{\phi}-\lambda) b(\hbar,\lambda,\cdot) = \mathcal{O}(e^{-\frac{\mathscr{C}}{\hbar}}) \hbox{ and } (P_{\phi}-\lambda) b(\hbar,\lambda,z_0) u(\hbar,\lambda,z)= \mathcal{O}(e^{-\frac{\mathscr{C}}{\hbar}}). \]
At $z_0$: $b(\hbar,\lambda,z_0) = b(\hbar,\lambda,z_0) u(\hbar,\lambda,z_0)+ \mathcal{O}(e^{-\frac{\mathscr{C}}{\hbar}})$. This proves that the formal symbol of $b(\hbar,\lambda,z)-b(\hbar,\lambda,z_0) u(\hbar,\lambda,z)$ is $0$ at $z_0$. Therefore by the uniqueness in Theorem \ref{JJ} we conclude that the formal symbols of $b(\hbar,\lambda,\cdot)$ and $b(\hbar,\lambda,z_0) u(\hbar,\lambda,\cdot)$ coincide for $\lambda,z$ near $\lambda_0,z_0$. This proves the result.
\end{proof}

The previous proposition allows us to build WKB quasimodes in $H_{\Phi_0}(\C)$ when $\lambda$ is $\mathcal{O}(\hbar)$ close to $\R$. Let us also choose some $M >0$ and define the set $\Sigma_\hbar := \mathrm{Neigh}(\lambda_0,\R) + i\hbar (-M,M)$ in the following section.
The main idea in that in \ref{lem:bscond} we will find than under the condition that $\lambda$ solves an equation of the type $\mu(\hbar,\lambda)e^{\frac{i}{\hbar}\mathcal{A}(\lambda)}$ with $\mu(\hbar,\lambda) = -1 + \mathcal{O}(\hbar)$ it is possible to prove the existence of WKB quasimodes. However, since a priori we only know the solutions of this equations are real modulo $\mathcal{O}(\hbar^2)$ (eventhough the eigenvalues are real), we prefer to restrict to a small neighbourhood of $\lambda_0$ where we expect to find such solutions. 

\begin{notation}
Note also that when we will write ``by reducing $\Sigma_\hbar$'', we will mean shortening $\mathrm{Neigh}(\lambda_0,\R)$ and considering smaller $M>0$. \\
\end{notation}

The small difficulty is that in the previous proposition, the pseudodifferential operator $\widetilde{P}$ acts on $H_{\Phi,(\lambda_0,z_0)}$ with $\Phi(\lambda,z) = - \mathrm{Im} \phi(\lambda,z)$ whereas we wish to build quasimodes in $H_{\Phi_0,z_0}$. The aim of the following construction is to remedy this issue. \\

Note that we expect that these quasimodes are localized near the energy sets $\mathscr{E}_{\mathrm{Re}(\lambda)}$ when $\lambda \in \Sigma_\hbar$. It will be useful to choose a smooth transversal \[ \mathrm{Neigh}(\lambda_0,\R) \ni \lambda \longmapsto z(\lambda)\hbox{ such that } \bigl(z(\lambda),\tfrac{2}{i} \partial_z \Phi_0(z(\lambda)) \bigr) \in \mathscr{E}_\lambda. \]

\begin{definition}
This allows to define a new phase function
\begin{equation}
\label{eq:chap4_defvarphi}
 \varphi(\lambda,z) = \int_{z(\mathrm{Re}(\lambda))}^z \zeta(\lambda,z) \mathrm{d} z - i \Phi_0(z(\mathrm{Re}(\lambda))).
 \end{equation}
\end{definition} 
 
In particular it satisfies $i\varphi\bigl(\lambda,z(\mathrm{Re}(\lambda))\bigr) = \Phi_0(z(\mathrm{Re}(\lambda)))$ corresponding to a function localized at $\mathscr{E}_{\mathrm{Re}(\lambda)}$. This will be clear thanks to Lemma \ref{lem:flambda} noticing that the $L^2_{\Phi_0}$ norm of $a(\hbar,\lambda,z)e^{\frac{i}{\hbar}\varphi(\lambda,z)}$ will be given by
\[ \left(\int_{\mathrm{Neigh}(z_0,\C)} e^{-2(\mathrm{Im} \varphi(\lambda,z) + \Phi_0(z))/\hbar} |a(\hbar,\lambda,z)|^2 \mathrm{d}L(z) \right)^{1/2}\, .\]
Therefore it will be useful to analyze the behaviour of $z \longmapsto \mathrm{Im} \varphi(\lambda,z) + \Phi_0(z)$.

\subsubsection{Refined analysis of the phase functions}
\label{sec:phase}
\begin{lemma}
\label{lem:flambda}
There exists $C >0$ such that, for all $\lambda \in \mathrm{Neigh}(\lambda_0,\R)$, for all $ z \in \mathrm{Neigh}(z_0,\C)$
\begin{equation}
\label{eq:flambda}
\frac{1}{C} \mathrm{d}(z,\mathscr{E}_\lambda)^2 \leq \mathrm{Im} \varphi(\lambda,z) + \Phi_0(z)  \leq C \mathrm{d}(z,\mathscr{E}_{\lambda})^2.
\end{equation}
\end{lemma} 
\begin{proof}
In what follows, we will lighten the notations by denoting $\varphi_\lambda$ instead of $\varphi(\lambda,\cdot)$ and define 
\[\forall z \in \mathrm{Neigh}(z_0,\C), \quad F(\lambda,\cdot):= \mathrm{Im} \varphi(\lambda,\cdot) + \Phi_0.\] 
In a nutshell, \eqref{eq:flambda} follows from the fact that for $\lambda \in \mathrm{Neigh}(\lambda_0,\R)$, $F_\lambda$ is strictly subharmonic and vanishes to second order on $\mathscr{E}_\lambda$ which is maximally totally real\footnote{The strict subharmonicity implies that $F_\lambda$ must grow quadratically in the normal direction to $\mathscr{E}_\lambda$ since the tangential direction correspond to the eigenvalue $0$ of the Hessian.}.

It remains to explain why $F_\lambda$ vanishes to second order on $\mathscr{E}_\lambda$. We have on $\mathscr{E}_\lambda \subset \Lambda_{\Phi_0}$
\begin{align*}
d \mathrm{Im} \varphi_\lambda & = \frac{i}{2}\left(\overline{\varphi_\lambda'(z)}d\overline{z} - \varphi_\lambda'(z)dz \right) = -\frac{i}{2}(\zeta dz - \overline{\zeta} d\overline{z}),
\end{align*} 
as well as, on $\Lambda_{\Phi_0}$
\begin{align*}
d\Phi_0 = \frac{1}{2}(\overline{z} dz + zd\overline{z}) = \frac{1}{2}(i\zeta dz  - i \overline{\zeta} d\overline{z}) = \frac{i}{2}(\zeta dz - \overline{\zeta} d\overline{z}).
\end{align*}
This proves that 
\begin{equation}
\label{eq:phi0}
\mathrm{d}(\mathrm{Im} \varphi_{\lambda}+\Phi_0)|_{\pi_z\mathscr{E}_{\lambda}} = 0,
\end{equation}
noticing that $\mathrm{Im}\varphi_{\lambda}+\Phi_0|_{z(\lambda)} = 0$ (this follows from Definition \ref{eq:chap4_defvarphi}) we infer that $\mathrm{Im} \varphi_{\lambda}+\Phi_0$ vanishes on $\pi_z \mathscr{E}_{\lambda}$.\\

Note that the constants $C$ in \eqref{eq:flambda} $\asymp$ can be chosen uniform in $\lambda$ thanks to the smoothness of $\mathrm{neigh}(\lambda_0,\C) \times \mathrm{Neigh}(\pi_z \mathscr{E}_{\lambda_0},\C) \ni (\lambda,z) \to F_\lambda(z)$.
\end{proof}

The previous estimate extends easily to $\Sigma_\hbar$ adding a $\mathcal{O}(\hbar)$ perturbation.

\begin{corollary}
\label{cor:blambda}
We find $C_1,C_2>0$ and a function $\mathcal{O}(\hbar)$ such that all $\lambda \in \Sigma_\hbar$, letting $\mu := \mathrm{Re}(\lambda)$ and for all $z \in \mathrm{Neigh}(z_0,\C)$
\begin{equation}
\label{eq:flambda2}
C_1\mathrm{d}(z,\mathscr{E}_{\mu})^2-\mathcal{O}(\hbar) \leq \bigl(\mathrm{Im} \varphi(\lambda,z) + \Phi_0(z) \bigr) \leq  C_2\mathrm{d}(z,\mathscr{E}_\mu)^2+\mathcal{O}(\hbar),
\end{equation}
\end{corollary}


\subsubsection{Changing the gauge}
In the following proposition, we prove that our constructions yields quasimodes in $H_{\Phi_0}(\mathrm{Neigh}(z_0,\C))$. \\

To do so, let us choose small neighbourhoods $\mathrm{Neigh}(z_0,\C)$, $\mathrm{Neigh}(\lambda_0,\C)$ on which the previous constructions hold.
\begin{proposition}
\label{prop:subhtoh}
Let $\mathrm{Neigh}(z_0, \C) \ni w \longmapsto a(\hbar,\lambda,w)e^{\frac{i}{\hbar}\varphi_\lambda(w)}$ be the quasimode constructed in the Proposition \ref{prop:chap3_qmodes}. Then, uniformly in $\lambda \in \Sigma_\hbar$,
\begin{equation}
(P_{c,\delta}-\lambda)(a(\hbar,\lambda,\cdot) e^{\frac{i}{\hbar}\varphi_\lambda}) = \mathcal{O}(e^{(\Phi_0-\mathscr{C})/\hbar}),
\end{equation} 
where $P_{c,\delta}$ is defined in \eqref{eq:approx}.
\end{proposition}
\begin{proof}
The conclusion is a consequence of \eqref{eq:bkw} and of the following result. There exists $\mathscr{C}>0$ such that, up to choosing $\delta>0$ small enough, for all $\lambda \in \Sigma_\hbar$, 
\begin{equation}
\label{eq:approx_analy}
(P_{c,\delta}-\lambda)(a(\hbar,\lambda,\cdot) e^{\frac{i}{\hbar}\varphi_\lambda}) = \mathcal{O}(e^{(\Phi_0-\mathscr{C}+\mathcal{O}(\hbar))/\hbar}) + (\widetilde{P}-\lambda) (a(\hbar,\lambda,\cdot)e^{\frac{i}{\hbar}\varphi_\lambda}),
\end{equation}
in a small neighbourhood of $(\lambda_0,z_0)$. \\
The term $P_{c,\delta}(a(\hbar,\cdot,\cdot) e^{\frac{i}{\hbar}\varphi_\lambda})$ is defined as an integral of a holomorphic $2$-form on the good contour
\[ \Lambda_{c,\delta} := \{ (w,\zeta) ~ | ~ \zeta = \tfrac{2}{i} \partial_z \Phi_0(\tfrac{z+w}{2}) +i c\overline{z-w} \},\]
for $-\mathrm{Im}(\zeta(z-w)) + \Phi_0(w)$. We notice that by Corollary \ref{cor:blambda} $-\mathrm{Im}\varphi_\lambda \leq \Phi_0 + \mathcal{O}(\hbar)$ as long as $\lambda \in \Sigma_\hbar$. Note also that $\widetilde{P}$ is defined as the integral of the same $2$-form than that of $P_{\Phi_0}$ but on the other contour
\[\mathscr{G}_{c,\delta}(\lambda,z) := \{ (w,\zeta) ~ | ~ \zeta = \partial_z \varphi(\lambda,w) +i c\overline{z-w}, ~|z-w| \leq \delta \}. \]
We can therefore apply Stokes lemma using the intermediary contours,
\[\mathscr{G}_{t,c,\delta}(\lambda,z) := \{(w,\zeta) ~ | ~  \zeta = (1-t)\partial_z \varphi(\lambda,w) + t \tfrac{2}{i} \partial_z \Phi_0(w) + ic\overline{z-w}, ~ |z-w| \leq \delta \}.\]
which are good contours for the phase 
\[ (w,\zeta) \longmapsto - \mathrm{Im}(\zeta(z-w)) + \Phi_t(\lambda,w),
\hbox{ with } \Phi_t(\lambda,w) = -(1-t) \mathrm{Im} \varphi(\lambda,w) + t \Phi_0(w). \]
We infer, by continuity in $(t,\lambda,z)$ that there exists $\mathscr{C}>0$ such that for all $\lambda \in \Sigma_\hbar$ and $t \in [0,1]$
\[-\mathrm{Im}\bigl(\zeta(z-w)) + \varphi(\lambda,w) \bigr)  \leq \Phi_0(z) - \mathscr{C} + \mathcal{O}(\hbar) \hbox{ on }  \{ (w,\zeta) \in \mathscr{G}_{t,C,\delta}(\lambda,z) ~ | ~ |z-w| = \delta \},\]
and this proves \eqref{eq:approx_analy}. To conclude the proof, using \eqref{eq:bkw} and \eqref{eq:approx_analy}, we infer that there exists $\mathscr{C}>0$ such that
\[ (P_{c,\delta}-\lambda)(a(\hbar,\cdot,\cdot) e^{\frac{i}{\hbar}\varphi_\lambda}) = \mathcal{O}(e^{(\Phi_0-\mathscr{C}+\mathcal{O}(\hbar))/\hbar}) + \mathcal{O}(e^{(-\mathrm{Im}\varphi_\lambda-\mathscr{C})/\hbar}) = \mathcal{O}(e^{\Phi_0-\mathscr{C}+\mathcal{O}(\hbar))/\hbar}). \]
Up to considering $\hbar$ sufficiently small this proves the result.
\end{proof}

We infer the existence of microlocal WKB quasimodes when $\lambda \in \Sigma_\hbar$.

\begin{proposition}
\label{prop:wkbineg}
There exist two neighbourhoods $\mathrm{Neigh}_1(z_0,\C) \Subset \mathrm{Neigh}_2(z_0,\C)$, $\mathscr{C}>0$ and a function $\mathcal{O}(e^{-\frac{\mathscr{C}}{\hbar}})$ such that for all $\lambda \in \Sigma_\hbar$,
\[ \|(P_{c,\delta} - \lambda) a(\hbar,\lambda,\cdot)e^{\frac{i}{\hbar}\varphi(\lambda,\cdot)} \|_{L^2_{\Phi_0}(\mathrm{Neigh}_1(z_0,\C))}\leq \mathcal{O}(e^{-\frac{\mathscr{C}}{\hbar}})\|a(\hbar,\lambda,\cdot)e^{\frac{i}{\hbar}\varphi(\lambda,\cdot)} \|_{L^2_{\Phi_0}(\mathrm{Neigh}_2(z_0,\C))}.\]
\end{proposition}
\begin{proof}
There exists $C>0$ such that for all $\lambda \in \Sigma_\hbar$, we have  
\begin{equation}
|a(\hbar,\lambda,w) e^{\frac{i}{\hbar}\varphi(\lambda,w)} | e^{-\Phi_0(w)/\hbar} \geq C e^{-\frac{C\mathrm{d}(w,\mathscr{E}_\mu)^2}{\hbar}}(|\partial_\zeta p_{\Phi_0}(w,-\varphi_{\lambda}'(w))|^{-1} + \mathcal{O}(\hbar^\frac12)).
\end{equation}
In particular we note that $\partial_\zeta p_{\Phi_0}(w,-\varphi_{\lambda}'(w))$ does not vanish for $(w,\lambda)$ in  \\ $\mathrm{Neigh}((\lambda_0,z_0),\C^2)$. The set
$\mathscr{U}_\hbar(\lambda) := \{ z \in \C ~ | ~ \mathrm{d}(z,\mathscr{E}_\mu) \leq \sqrt{\hbar} \} \cap \mathrm{Neigh}(z_0,\C)$, with $\mu = \mathrm{Re}(\lambda)$, has volume of order $\hbar^{1/2}$. We deduce that \[ \| a(\hbar,\lambda,\cdot) e^{\frac{i}{\hbar}\varphi(\lambda,\cdot)} \|_{L^2_{\Phi_0}(\mathrm{Neigh}(z_0,\C))}^2 \gtrsim \hbar^{\frac12} \hbox{uniformly in } \lambda \in \Sigma_\hbar, \]
and conclude thanks to Proposition \ref{prop:subhtoh}.
\end{proof}

\subsection{Patching WKB expansions with a $\overline{\partial}$ Lemma}
The goal of this section is to patch the microlocal WKB quasimodes constructed above and prove Proposition \ref{lem:bscond}. 
\label{sec:patchwkbchap3} 

\subsubsection{$\overline{\partial}$ Lemma adapted to WKB expansions}
In order to patch the WKB expansions on a small neighbourhood of $\mathscr{E}_{\lambda_0}$ we will use Hörmander's $\overline{\partial}$ lemma that we recall in what follows. 

\begin{lemma}[Hörmander $\overline{\partial}$ Lemma]
\label{lem:hodbar}
Let $f \in L^2_{\Phi_0}(\C,\mathfrak{m})$. Then there exists $u \in L^2_{\Phi_0}(\C,\mathfrak{m})$ such that for $\hbar$ small enough
\begin{equation}
\overline{\partial}_z u =f  \quad \hbox{ and } \quad \| u \|_{L^2_{\Phi_0}(\C,\mathfrak{m})} \leq 2\hbar \| f \|_{L^2_{\Phi_0}(\C,\mathfrak{m})}.
\end{equation}
\end{lemma}
\begin{proof}
This is a straightforward consequence of Hörmander's Lemma \cite[15.1.1]{Ho} with weight $\frac{\Phi_0}{\hbar} + \log(\mathfrak{m})$.
\end{proof}

We will adapt the $\overline{\partial}$ lemma to glue WKB quasimodes in the following setting. Let us consider a family of complex balls $(\mathbb{B}_j)_{j \in [\![0,m]\!]}$ that we suppose satisfy the following.
\begin{enumerate}[label=A.\arabic*]
\item \label{hyp:ouv1} For all $j \in [\![0,m]\!]$, the ball $\mathbb{B}_j$ is centered at some point $z_j \in \mathscr{E}_{\lambda_0}$.
\item \label{hyp:ouv2} The intersection $\mathbb{B}_j \cap \mathbb{B}_k$ is empty if and only if $k \notin [\![j-1,j+1]\!]$ (denoting $\mathbb{B}_{m+1} = \mathbb{B}_0$).
\item \label{hyp:ouv3} The Hamiltonian flow of $z_0$ crosses $\mathbb{B}_0$, then $\mathbb{B}_1$, $\mathbb{B}_2$ all the way back to $\mathbb{B}_0$ in the increasing order of indices.
\end{enumerate}

Let $\Omega \Subset \bigcup_{j=0}^m \mathbb{B}_j$ be a connected open set containing $\mathscr{E}_{\lambda_0}$. The standard results on partition of unity gives us a family of smooth functions $(\chi_j)_{j \in [\![0,m]\!]}$ satisfying that for all $j \in [\![0,m]\!]$, $\chi_j \in \mathscr{C}_c^{\infty}(\C)$, $\mathrm{supp} ~ \chi_j \subset \mathbb{B}_j$ and $\sum_{j=0}^m \chi_j = 1$ on $\Omega$. This allows us to state the following Lemma in the spirit of \v Cech cohomology.

\begin{lemma}[Patching Lemma in $L^2_{\Phi_0}(\C,\mathfrak{m})$]
\label{lem:cech}

For each $j \in [\![0,m]\!]$, we consider $f_j, r_j \in \mathrm{Hol}(\mathbb{B}_j)$ satisfying
\begin{equation}
\label{eq:assumpcech}
\forall j\in [\![0,m]\!], ~ \exists c_{j,j+1} \in \C^*, \quad c_{j,j+1} f_{j+1} =  f_{j} + r_j \hbox{ on } \mathbb{B}_j \cap \mathbb{B}_{j+1},
\end{equation}
where we denote $\mathbb{B}_{m+1} = \mathbb{B}_0$ and $f_{m+1} = f_0$.

There exists a constant $C>0$ (depending on the partition of unity) such that the following holds. If the condition $\prod_{j=0}^{m} c_{j,j+1} = 1$ is satisfied then there exists $u \in L^2_{\Phi_0}(\C,\mathfrak{m})$ such that, denoting $d_0 = 1, d_j = c_{0,1} \cdots c_{j-1,j}$, we have 
\begin{equation}
\label{eq:cech}
\sum_{j=0}^m d_j \chi_j f_j - u \in H_{\Phi_0}(\C,\mathfrak{m}) \hbox{ and } \|u\|_{L^2_{\Phi_0}(\C,\mathfrak{m})} \leq C \hbar \sum_{j=0}^m \left(\| d_j r_j \|_{L^2_{\Phi_0}(\mathbb{B}_j)} + \|d_j f_j \|_{L^2_{\Phi_0}(\mathbb{B}_j \setminus \Omega)} \right).
\end{equation}
\end{lemma}
\begin{proof}
 We consider $g := \sum_{j = 0}^m d_j \chi_j f_j$ and notice that 
 
 --- on $\mathbb{B}_j \cap \mathbb{B}_{j+1}\cap \Omega$, $\overline{\partial} g = \overline{\partial}\chi_j(d_j f_j - d_{j+1} f_{j+1}) = -d_j \overline{\partial}\chi_j r_j $,
 
--- on $\C$, $\overline{\partial} g = \sum_{j=0}^m d_j \overline{\partial}\chi_j  f_j$.

This gives some $c>0$ only depending on the partition of unity
\begin{equation}
\| \overline{\partial} g \|_{L^2_{\Phi_0}(\C,\mathfrak{m})} \leq c \sum_{j=0}^m \left(\| d_j r_j \|_{L^2_{\Phi_0}(\mathbb{B}_j)} + \| d_j f_j \|_{L^2_{\Phi_0}(\mathbb{B}_j \setminus \Omega)} \right).
\end{equation}
We apply Hörmander $\overline{\partial}$ Lemma \ref{lem:hodbar} giving $u \in L^2_{\Phi_0}(\C,\mathfrak{m})$ such that
\begin{equation}
g-u \in H_{\Phi_0}(\C,\mathfrak{m}) \hbox{ and } \| u \|_{L^2_{\Phi_0}(\C,\mathfrak{m})} \leq 2c \hbar \sum_{j=0}^m \left(\|d_j r_j \|_{L^2_{\Phi_0}(\mathbb{B}_j)} + \| d_j f_j \|_{L^2_{\Phi_0}(\mathbb{B}_j \setminus \Omega)} \right).
\end{equation}
\end{proof}

From now on the energy sets $\mathscr{E}_\lambda$ will be identified with their projections on the $z$ variable thanks to the fact that \eqref{eq:proj} is an isomorphism. 

\subsubsection{Construction of the quasimodes}
\label{sec:qmconstr}
In what follows, we apply the $\overline{\partial}$ lemma and prove the following proposition. The undetermincacy of the sign of the principal symbol of $\mu(\hbar,\lambda)$ will be removed in the next section.
\begin{proposition}
\label{lem:bscond}
There exist $\mathscr{C} >0$, a function $\mathcal{O}(e^{-\frac{\mathscr{C}}{\hbar}})$ and a classical analytic symbol $\mu(\hbar,\lambda)=\pm 1 + \mathcal{O}(\hbar)$ such that the following holds. For all $\lambda$ in $\Sigma_\hbar$ and satisfying
\begin{equation}
\label{eq:bs} \tag{BS}
\mu(\hbar,\lambda)e^{\frac{i}{\hbar}\mathcal{A}(\lambda)} = 1,
\end{equation}
there exists $\Psi_\lambda \in H_{\Phi_0}(\C)$ such that
\begin{equation}
\|(P_{\Phi_0}-\lambda) \Psi_\lambda \|_{L^2_{\Phi_0}(\C)} \leq \mathcal{O}(e^{-\frac{\mathscr{C}}{\hbar}}) \| \Psi_\lambda \|_{L^2_{\Phi_0}(\C)}.
\end{equation}
\end{proposition}

\begin{figure}
\begin{tikzpicture}
  \tikzset{dotted circle/.style={draw, dotted, thick}}

  \def\r{1}

  \coordinate (A) at (0,0.3);
  \coordinate (B) at (1.5,0);
  \coordinate (C) at (3,0.3);

  \draw[dotted circle] (A) circle (\r);
  \draw[dotted circle] (B) circle (\r);
  \draw[dotted circle] (C) circle (\r);

\draw (0,1.3) node[above] {$\mathbb{B}_0$};
\draw (1.5,1) node[above] {$\mathbb{B}_2$};
\draw (3,1.3) node[above] {$\mathbb{B}_3$};

  \draw[thick, blue]
 (-2,0.5) .. controls (2,-0.2) .. (5,1) ;
\draw[blue] (-1.9,0.6) node[left] {$\mathscr{E}_{\lambda_0}$}; 
\fill[red, opacity = 0.2] plot [smooth cycle, thick, red] coordinates
 {(-2,0.7) (2,0.5) (5,1.3) (5,0.5) (2,-0.5) (-2,0.2)};  
\draw[red, opacity = 0.8] (5,0.5) node[right] {$\Omega$};
 \fill[red, opacity = 0.5] plot [smooth cycle, thick, red] coordinates
 {(-0.7,0.4) (0.8,0.3) (0.7,-0.25) (-0.7,0)};  
 \fill[red, opacity = 0.5] plot [smooth cycle, thick, red] coordinates
 {(0.8,0.3) (2.3,0.3) (2.3,-0.25) (0.8,-0.1)}; 
  \fill[red, opacity = 0.5] plot [smooth cycle, thick, red] coordinates
 {(2.3,0.4) (3.8,0.65) (3.8,0.3) (2.35,-0.05)};
 \draw[red] (0,-0.2) node[below] {$K_0$}; 
 
\draw[red] (1.5,-0.3) node[below] {$K_1$}; 

 \draw[red] (3,-0.1) node[below] {$K_2$};

\end{tikzpicture}
\caption{Illustration of the patching procedure.}
\label{fig:01}
\end{figure}

As described in Figure \ref{fig:01} and thanks to Proposition \ref{prop:chap3_qmodes} we find a family of $m$ balls, $(\mathbb{B}_j)_{j \in [\![0,m]\!]}$, an open set $\Omega \Subset \bigcup_{j=0}^m \mathbb{B}_j$ and compact sets $(K_j)_{j \in  [\![0,m]\!]}$ such that \[K_j \subset \Omega \cap \mathbb{B}_j, ~ \mathscr{E}_{\lambda_0} \subset \bigcup_{j=0}^m K_j \subset \Omega \subset \bigcup_{j=0}^m \mathbb{B}_j \]
and satisfying the following.
\begin{enumerate}[label= \textit{I.\arabic*}, ref=\textit{I.\arabic*}]
\item \label{hyp:semiglo1} The assumptions \ref{hyp:semiglo1}, \ref{hyp:semiglo2}, \ref{hyp:semiglo3} are verified. 
\item \label{hyp:semiglo2} There exists a smooth transversal $\lambda \longmapsto z_j(\lambda)$ satisfying $z_j(\lambda) \in \mathscr{E}_{\mathrm{Re}(\lambda)}$.
\item \label{hyp:semiglo3} For all $j \in [\![0,m]\!]$ the phase $\varphi_j(\lambda,\cdot)$ is well defined on $\mathbb{B}_j$ by the construction of Section \ref{sec:phase}. It satisfies $i\varphi_j(\lambda,z_j(\lambda)) = \Phi_0(z_j(\lambda))$ and also Lemma \ref{lem:flambda}.
\item \label{hyp:semiglo4} There exists a classical analytic symbol $a_j(\hbar,\lambda,z)$ well defined for $\lambda$ near $\lambda_0$, $z$ in $\mathbb{B}_j$ and satisfying 
\begin{equation}
\|(P_{c,\delta} - \lambda) a(\hbar,\lambda,\cdot)e^{\frac{i}{\hbar}\varphi(\lambda,\cdot)} \|_{L^2_{\Phi_0}(K_j \cap \mathbb{B}_j)}\leq \mathcal{O}(e^{-\frac{\mathscr{C}}{\hbar}})\|a(\hbar,\lambda,\cdot)e^{\frac{i}{\hbar}\varphi(\lambda,\cdot)} \|_{L^2_{\Phi_0}(\mathbb{B}_j)}.
\end{equation}
\item The principal symbol of $a_j(\hbar,\lambda,\cdot)$ is non zero and satisfies \[ a_j^0(\lambda,z)^2= \partial_\zeta p_{\Phi_0}(z,\zeta(\lambda,z))^{-1}. \]
\end{enumerate}

The last statement of Proposition \ref{prop:chap3_qmodes} guarantees that for all $j \in  [\![0,m]\!]$ there exists a classical analytic symbol $\alpha_{j,j+1}(\hbar,\lambda)$ with principal symbol $\pm1$ satisfying for all $\lambda$ close enough to $\lambda_0$ (denoting $\mathbb{B}_{m+1} = \mathbb{B}_0$)
\begin{equation}
\label{eq:transition}
\alpha_{j,j+1}(\hbar,\lambda) a_{j+1}(\hbar,\lambda,\cdot) = a_j(\hbar,\lambda,\cdot) + \varepsilon_j(\hbar,\lambda,\cdot) \hbox{ on } \mathbb{B}_j \cap \mathbb{B}_{j+1}.
\end{equation}
where there is $\mathscr{C}>0$ such that $|\varepsilon_j| \leq \mathcal{O}(e^{-\frac{\mathscr{C}}{\hbar}})$.

Let us start with a brief look at the phase functions.

\begin{lemma}
\label{lem:partaction}
For all $j \in  [\![0,m]\!]$, there exists $\mathcal{A}_{j,j+1}(\lambda) \in \C$ such that for all $\lambda \in \Sigma_\hbar$, 
\begin{equation}
\label{eq:cst}
\begin{aligned}
\varphi_j(\lambda,\cdot) - \varphi_{j+1}(\lambda,\cdot) = \mathcal{A}_{j,j+1}(\lambda) \hbox{ on } \mathbb{B}_j \cap \mathbb{B}_{j+1}, \sum_{j=1}^m \mathcal{A}_{j,j+1}(\lambda) = \mathcal{A}(\lambda).
\end{aligned}
\end{equation}
and $\mathrm{Im}(\mathcal{A}_{j,j+1}(\lambda)) = \mathcal{O}(\hbar)$.
\end{lemma}
\begin{proof}
The fact that $\mathbb{B}_j \cap \mathbb{B}_{j+1}$ is connected and that the maps $\varphi_{j+1}(\lambda,\cdot)$, $\varphi_j(\lambda,\cdot)$ have same derivative guarantees the existence of $\mathcal{A}_{j,j+1}(\lambda)$ satisfying \eqref{eq:cst}.
 In particular $\mathcal{A}_{j,j+1}(\lambda)$ is a partial action in the sense that defining $\gamma_{j,j+1}$ as the simple path from $z_j(\lambda)$ to $z_{j+1}(\lambda)$ in $\pi_z \mathscr{E}_{\mathrm{Re}(\lambda)}$,  we have from the identity $\partial_z\varphi_j(\hbar,\lambda,z) = \zeta(\lambda,z)$,
\begin{equation} 
\mathcal{A}_{j,j+1}(\lambda) = \int_{\gamma_{j,j+1}} \zeta(\lambda,z) dz + i \cdot \mathrm{cobord} \quad \hbox{ where } \mathrm{cobord} = \Phi_0(z_{j}(\lambda)) - \Phi_0(z_{j+1}(\lambda)). 
\end{equation}
The concatenation of all the paths $\gamma_{j,j+1}$ is a simple loop going through $\mathscr{E}_{\lambda}$. Since the orientation of the loops $\gamma_{j,j+1}$ are given by the Hamiltonian flow, we conclude that $\sum_{j=0}^m \mathcal{A}_{j,j+1}(\lambda) = \int_{\mathscr{E}_{\mathrm{Re}(\lambda)}} \zeta(\lambda,z) dz = \mathcal{A}(\lambda)$. Eventually we deduce from Corollary \ref{cor:blambda} that $\mathrm{Im}(\mathcal{A}_{j,j+1}(\lambda)) = \mathcal{O}(\hbar)$.
\end{proof}

We now apply the patching Lemma \ref{lem:cech} to get the following lemma, denoting $\mathcal{A}_j(\lambda) = \mathcal{A}_{0,1}(\lambda) + \cdots + \mathcal{A}_{j-1,j}(\lambda)$, $\alpha_j(\hbar,\lambda) = \alpha_{0,1}(\hbar, \lambda) \cdots \alpha_{j-1,j}(\hbar,\lambda)$ and
\[ \psi_\lambda := \sum_{j=0}^m \chi_j(\cdot) \alpha_j(\hbar,\lambda) a_j(\hbar,\lambda,\cdot)e^{\frac{i}{\hbar}(\varphi_j(\lambda,\cdot)+ \mathcal{A}_j(\lambda))}. \] 

\begin{lemma}
\label{lem:cechbs}
There exists a classical analytic symbol $\mu(\hbar,\lambda)$ whose principal symbol is $\pm 1$ such that the following holds. There exist $\mathscr{C} >0$ and a function $\mathcal{O}(e^{-\frac{\mathscr{C}}{\hbar}})$ such that  if
$
\mu(\hbar,\lambda)e^{\frac{i}{\hbar}\mathcal{A}(\lambda)} = 1
$
for $\lambda \in \Sigma_\hbar$ we find $u_\lambda \in L^2_{\Phi_0}(\C,\mathfrak{m})$ satisfying
\begin{equation}
 \psi_\lambda  - u_\lambda \in H_{\Phi_0}(\C,\mathfrak{m}) \hbox{ and }  \|u_\lambda\|_{L^2_{\Phi_0}(\C,\mathfrak{m})} \leq \mathcal{O}(e^{-\frac{\mathscr{C}}{\hbar}}).
\end{equation}
\end{lemma}
\begin{proof}
We apply the patching Lemma \ref{lem:cech} (considering $\lambda$ as a parameter) with 
\[f_j = a_j(\hbar,\lambda,\cdot)e^{\frac{i}{\hbar}\varphi_j(\lambda,\cdot)}, ~ c_{j,j+1}(\hbar,\lambda) = e^{\frac{i}{\hbar}\mathcal{A}_{j,j+1}(\lambda)}\alpha_{j,j+1}(\hbar,\lambda)\,, \] $d_j(\hbar,\lambda) = \alpha_j(\hbar,\lambda)e^{\frac{i}{\hbar}\mathcal{A}_j(\lambda)}$ and $r_j(\hbar,\lambda,z):= \varepsilon_j(\hbar,\lambda,z)e^{\frac{i}{\hbar}\varphi_j(\lambda,z)}$, with $\epsilon_j$ defined in \eqref{eq:transition}. Setting $\mu(\hbar,\lambda) := \alpha_{0,1}(\hbar,\lambda) \alpha_{1,2}(\hbar,\lambda) \cdots \alpha_{m,0}(\hbar,\lambda),$ thanks to Lemma \ref{lem:partaction} the condition $\prod_{j=0}^m c_{j,j+1} = 1$ in Lemma \ref{lem:cech} is equivalent to
$\mu(\hbar,\lambda)e^{\frac{i}{\hbar}\mathcal{A}(\lambda)} = 1$. As soon as this condition is satisfied we obtain $u_\lambda \in H_{\Phi_0}(\C,\mathfrak{m})$ satisfying $\psi_\lambda - u_\lambda \in H_{\Phi_0}(\C,\mathfrak{m})$ and 
\begin{equation}
\label{eq:estdbar}
\|u_\lambda\|_{L^2_{\Phi_0}(\C,\mathfrak{m})} \leq C \sum_{j=0}^m \bigl(\| d_j(\hbar,\lambda) r_j \|_{L^2_{\Phi_0}(\mathbb{B}_j)} + \| d_j(\hbar,\lambda) f_j \|_{L^2_{\Phi_0}(\mathbb{B}_j \setminus \Omega)}\bigr).
\end{equation}
 The assumption on the principal symbol of $a_{j+1}(\hbar,\lambda,\cdot)$ guarantees that $|\alpha_{j,j+1}(\hbar,\lambda)| = 1 + \mathcal{O}(\hbar)$ (uniformly in $\lambda$ near $\lambda_0$) thus $\mu(\hbar,\lambda) = \pm 1 + \mathcal{O}(\hbar)$. \\
  Using Lemma \ref{lem:partaction} we recall that $\mathrm{Im}(\mathcal{A}_j(\lambda)) = \mathcal{O}(\hbar)$ uniformly in $\lambda \in \Sigma_\hbar$ and $d_j(\hbar,\lambda) = \mathcal{O}(1)$. Possibly restricting $\Sigma_\hbar$, we infer again by Corollary \ref{cor:blambda} that for some $\mathscr{C}>0$, uniformly in $\lambda \in \Sigma_\hbar$
  \[ \| d_j(\hbar,\lambda) r_j(\hbar,\lambda,\cdot) \|_{L^2_{\Phi_0}(\mathbb{B}_j)}, \|d_j(\hbar,\lambda) f_j(\hbar,\lambda,\cdot) \|_{L^2_{\Phi_0}(\mathbb{B}_j \setminus \Omega)} \leq \mathcal{O}(e^{-\frac{\mathscr{C}}{\hbar}}).\] Eventually from \eqref{eq:transition} we conclude that there exists (another) $\mathscr{C}>0$ such that the right hand side of \eqref{eq:estdbar} is $\mathcal{O}(e^{-\frac{\mathscr{C}}{\hbar}}).$
\end{proof}

The map $\Psi_{\lambda} := \psi_\lambda - u_\lambda \in H_{\Phi_0}(\C,\mathfrak{m})$ defined in the previous lemma will be our candidate in the quest for quasimodes. Before providing the proof of Proposition \ref{lem:bscond} let us provide some useful estimates. We define $K  = \bigcup_{j=0}^m K_j \Subset \Omega$ that is a compact neighbourhood of $\mathscr{E}_{\lambda_0}$. 

\begin{lemma}[Useful Estimates] \label{lem:estpsi}
Possibly restricting $\Sigma_\hbar$, there exist $\mathscr{C},C>0$ and a function $\mathcal{O}(e^{-\frac{\mathscr{C}}{\hbar}})$ such that for all $\lambda \in \Sigma_\hbar$  satisfying \eqref{eq:bs}: 
\begin{enumerate}
\item[i)] $ \| (P_{c,\delta} - \lambda) \psi_\lambda \|_{L^2_{\Phi_0}(K)} \leq \mathcal{O}(e^{-\frac{\mathscr{C}}{\hbar}})\| \psi_\lambda \|_{L^2_{\Phi_0}(\C)} \hbox{ and } \| \psi_\lambda \|_{L^2_{\Phi_0}(\C)} \geq C \hbar^{\frac14}.$
\item [ii)] $ \|\psi_\lambda \|_{L^2_{\Phi_0}(\complement K,\mathfrak{m})} \leq \mathcal{O}(e^{-\frac{\mathscr{C}}{\hbar}}) $
\item[iii)] $\|  \Psi_\lambda \|_{L^2_{\Phi_0}(\C)} \asymp \| \Psi_\lambda \|_{L^2_{\Phi_0}(\C,\mathfrak{m})}$ and $\|  \Psi_\lambda \|_{L^2_{\Phi_0}(\C)} \asymp \| \psi_\lambda \|_{L^2_{\Phi_0}(\C)}$ (with constants uniform in $\lambda$).
\end{enumerate}
\end{lemma}
\begin{proof}
The estimates are elementary consequences of the previous lemmas. However for the sake of completeness we choose to provide some details to the reader.

\textit{i)} By construction we have $K \subset \bigcup_{j \in [\![0,m]\!]} \mathbb{B}_j$ and by Proposition \ref{prop:wkbineg}, on each $\mathbb{B}_j$, 
\[ \|(P_{c,\delta} - \lambda) a_j(\hbar,\lambda,\cdot)e^{\frac{i}{\hbar}\varphi_j(\lambda,\cdot)} \|_{L^2_{\Phi_0}(\mathbb{B}_j \cap K)}\leq \mathcal{O}(e^{-\frac{\mathscr{C}}{\hbar}})\|a_j(\hbar,\lambda,\cdot)e^{\frac{i}{\hbar}\varphi_j(\lambda,\cdot)} \|_{L^2_{\Phi_0}(\mathbb{B}_j)}.\]
The estimate on $\| \psi_\lambda \|_{L^2_{\Phi_0}(\C)}$ stems from the proof of Proposition \ref{prop:wkbineg}.

\textit{ii)} To prove $\|\chi_j d_j(\hbar,\lambda) a_j(\hbar,\lambda,\cdot) e^{\frac{i}{\hbar}\varphi_j}\|_{L^2_{\Phi_0}(\complement K,\mathfrak{m})} \leq \mathcal{O}(e^{-\mathscr{C}/{\hbar}})$, up to shortening $\Sigma_\hbar$, we have that for all $\lambda \in \Sigma_\hbar$, $d_j(\hbar,\lambda) = \mathcal{O}(1)$ and by Corollary \ref{cor:blambda} $-(\mathrm{Im}\varphi_j(\lambda,\cdot) +\Phi_0) \leq -\mathscr{C}$ for some $\mathscr{C}>0$ and $\hbar$ small enough. This proves \textit{ii}) by the triangular inequality.

\textit{iii)} To prove the last statement, we notice that $\|  \Psi_\lambda \|_{L^2_{\Phi_0}(\C)} \asymp \| \psi_\lambda \|_{L^2_{\Phi_0}(\C)}$. The map $\psi_\lambda$ has compact support in $\bigcup_{j=0}^m \mathbb{B}_j$ and we can conclude thanks to the equivalence $\| \cdot \|_{L^2_{\Phi_0}(\bigcup_{j=0}^m \mathbb{B}_j)} \asymp \| \cdot \|_{L^2_{\Phi_0}(\bigcup_{j=0}^m \mathbb{B}_j,\mathfrak{m})}$.
\end{proof}

\begin{proof}[Proof of Proposition \ref{lem:bscond}]
Let us then gather all the elements and prove Proposition \ref{lem:bscond}. The idea is to apply the technical Lemma \ref{lem:loca} and consider $K_1 \subset \overset{\circ}{K}$ another compact neighbourhood of $\mathscr{E}_{\lambda_0}$ to deal with the border effects. \\

\textit{Norm estimate on $\complement K$}: By the estimates of Lemma \eqref{lem:estpsi} we know that there exist $\mathscr{C} >0$ and a function $\mathcal{O}(e^{-\frac{\mathscr{C}}{\hbar}})$, possibly restricting $\Sigma_\hbar$, for all $\lambda \in \Sigma_\hbar$ satisfying \eqref{eq:bs},
 \[ \| \psi_{\lambda} \|_{L^2_{\Phi_0}(\complement K_1,\mathfrak{m})} \leq \mathcal{O}(e^{-\frac{\mathscr{C}}{\hbar}})\| \psi_{\lambda} \|_{L^2_{\Phi_0}(\C,\mathfrak{m})}. \]
 Lemma \eqref{lem:loca} and the previous inequality yield, changing the value of $\mathscr{C} >0$ at each line if need be, 
 \begin{equation}
 \begin{aligned}
 \| (P_{\Phi_0}-\lambda) \Psi_{\lambda} \|_{L^2_{\Phi_0}(\complement K)} & \leq \rho \| \Psi_{\lambda} \|_{L^2_{\Phi_0}(\complement K_1,\mathfrak{m})} + \mathcal{O}(e^{-\frac{\mathscr{C}}{\hbar}}) \| \Psi_{\lambda} \|_{L^2_{\Phi_0}(\C,\mathfrak{m})} \\ & \leq \mathcal{O}(e^{-\frac{\mathscr{C}}{\hbar}}) \| \Psi_{\lambda} \|_{L^2_{\Phi_0}(\C,\mathfrak{m})}.
 \end{aligned}
 \end{equation}
\textit{Norm estimate on $K$:} We take $P_{c,\delta}$ defined in \eqref{eq:approx} with $\delta$ smaller than $\mathrm{dist}(K, \partial \Omega)$. The goal is to control the terms of the right hand side of
 \begin{multline}
 \label{eq:ctr}
\| (P_{\Phi_0}-\lambda) \Psi_{\lambda} \|_{L^2_{\Phi_0}( K)} \\ \leq  \| (P_{c,\delta}-\lambda) \psi_{\lambda} \|_{L^2_{\Phi_0}( K)} + \| P_{\Phi_0} - P_{c,\delta}\|_{L^2_{\Phi_0}(\C,\mathfrak{m}) \to L^2_{\Phi_0}(\C)} \|\Psi_{\lambda} \|_{L^2_{\Phi_0}(\C,\mathfrak{m})}+ C\| u_\lambda \|_{L^2_{\Phi_0}(\C,\mathfrak{m})}.
\end{multline}
The term $\| u_\lambda \|$ is easily controled thanks to the norm estimate of the $\overline{\partial}$ lemma. We also have by Lemma \ref{lem:rem} that 
\[ \| P_{\Phi_0} - P_{c,\delta}\|_{L^2_{\Phi_0}(\C,\mathfrak{m}) \to L^2_{\Phi_0}(\C)} \leq \mathcal{O}(e^{-\frac{\mathscr{C}}{\hbar}}),\] for some $\mathscr{C} >0$ and by the first estimate of Lemma \eqref{lem:estpsi} for another $\mathscr{C} >0$, 
\begin{equation}
\| (P_{c,\delta} - \lambda) \psi_{\lambda} \|_{L^2_{\Phi_0}(K)} \leq \mathcal{O}(e^{-\frac{\mathscr{C}}{\hbar}}) \| \psi_{\lambda} \|_{L^2_{\Phi_0}(\C)} \leq \mathcal{O}(e^{-\frac{\mathscr{C}}{\hbar}}) \| \Psi_{\lambda} \|_{L^2_{\Phi_0}(\C)}.
\end{equation}
We can conclude this estimate using the point \textit{iii)} of the previous lemma and that fact that $\| \psi_\lambda \|_{L^2_{\Phi_0}(\C)}$ is bounded by below by $c\hbar^{\frac14}$ for some $c>0$.
\end{proof}

\subsection{Maslov's correction}
\label{sec.maslov}

\begin{proposition}
We can compute the first order of $\mu(\hbar,\lambda)$ which is related to Maslov's correction. It is given by
\label{lem:msl}
\begin{equation}
\label{eq:msl}
\mu(\hbar,\lambda) = -1 + \mathcal{O}(\hbar).
\end{equation}
\end{proposition}
\begin{proof}
First, letting $j \in [\![0,m]\!]$, as defined in Section \ref{sec:qmconstr} the principal symbol $a_j^0(\lambda,z)$ of $a_j(\hbar,\lambda,\cdot)$ is a local squareroot of $\partial_\zeta p_{\Phi_0}(z,\zeta(\lambda,z))^{-1}$. Possibly recursively replacing $a_j(\hbar,\lambda,z)$ by its opposite we can assume that for all $j \in [\![0,m-1]\!]$, $a_j^0(\lambda,z) = a_{j+1}^0(\lambda,z)$ on $\mathbb{B}_j \cap \mathbb{B}_{j+1}$ and for all $\lambda$ near $\lambda_0$. This yields that with $\alpha_{j,j+1}(\hbar,\lambda)$ defined in \eqref{eq:transition} , $\alpha_{j,j+1}(\hbar,\lambda) = 1 + \mathcal{O}(\hbar)$  for $j=1,2,\cdots,m-1$. Now we prove that necesarily $\alpha_{m,0}(\hbar,\lambda) = -1 + \mathcal{O}(\hbar)$. Indeed assume that this is not the case, that is to say $\alpha_{m,0}(\hbar,\lambda) = 1 + \mathcal{O}(\hbar)$. In other words we obtain
\[\forall j \in [\![0,m]\!], \quad a_j^0(\lambda,z) = a_{j+1}^0(\lambda,z) \hbox{ on } \mathbb{B}_j \cap \mathbb{B}_{j+1}\]
with the convention $a_{m+1}^0(\lambda,z) = a_0^0(\lambda,z)$ and $\mathbb{B}_0 = \mathbb{B}_{m+1}$. We can then apply the usual patching result of \v Cech cohomology for holomorphic functions providing a holomorphic function $b(\lambda,z)$ fitting with all the $a_j^0(\lambda,z)$ on the union of $\mathbb{B}_j$. In particular $b(\lambda,z)$ satisfies $b(\lambda,z)^2 = \partial_\zeta p_{\Phi_0}(z,\zeta(\lambda,z))^{-1}$ and we infer that $\partial_\zeta p_{\Phi_0}(z,\zeta(\lambda,z))$ has a holomorphic squareroot defined on a neighbourhood of $\mathscr{E}_{\lambda_0}$. This contradicts the result of Lemma \ref{lem:rac}. We conclude that $\alpha_{m,0}(\hbar,\lambda) = -1 + \mathcal{O}(\hbar)$ and 
\begin{equation}
\mu(\hbar,\lambda) = \alpha_{0,1}(\hbar,\lambda)  \alpha_{1,2}(\hbar,\lambda) \cdots  \alpha_{m-1,m}(\hbar,\lambda)  \alpha_{m,0}(\hbar,\lambda) = -1 + \mathcal{O}(\hbar).
\end{equation}
\end{proof}

At order $2$ it is then easy to see that (with a small abuse of notations) the equation \eqref{eq:bs} becomes 
\begin{equation}
\label{eq:appmsl}
 \mathcal{A}(\lambda) + \pi \hbar + \mathcal{O}(\hbar^2) \in 2\pi\hbar \Z.
 \end{equation}
We obtain that the solutions of the equation \eqref{eq:bs} near $\lambda_0$ belong to $\Sigma_\hbar$. This enables the existence of exponentially sharp quasimodes by Proposition \ref{lem:bscond}. The term $\pi \hbar$ in the left hand side of \eqref{eq:appmsl} is deduced from Proposition \ref{lem:msl} and usually known as Maslov's correction.

\begin{remark}[Interpretation of Maslov's correction]
\label{rk:maslov}
From another perspective, by the previous construction, we can choose a phase $\varphi_{\lambda}$ on the universal cover of $\Omega$ as a primitive of $\zeta(\lambda,z)$ and notice that the ``principal symbol'' $\partial_\zeta p_{\Phi_0}(z,\zeta(\lambda,z))^{-\frac12}$ is also well defined on the universal cover (as it is simply connected). If we truncate at order $2$ in $\hbar$, the condition \eqref{eq:bs} becomes $\mathcal{A}(\lambda)+\pi \hbar \in 2\pi \hbar \Z$ and is equivalent to the condition at which $\partial_\zeta p_{\Phi_0}(z,\zeta(\lambda,z))^{-\frac12} e^{\frac{i}{\hbar}\varphi_\lambda(z)}$ defines a holomorphic function on $\Omega$. It is interesting to notice that $\partial_\zeta p_{\Phi_0}(z,\zeta(\lambda,z))^{-\frac12} e^{\frac{i}{\hbar}\varphi_\lambda(z)}$ is well defined ``semi-globally'' near $\mathscr{E}_{\lambda_0}$ even though individually the symbol and the phase are not.
\end{remark}

\subsection{Inversion in the Class of Analytic Symbols and Conclusion}
\textit{The goal of this section is to invert \eqref{eq:bs} in the class of analytic symbols to finish the proof of Theorem \ref{th1} in a small neighbourhood of $\lambda_0$.}

We look for a classical analytic symbol $\lambda(\h,I) = \mathcal{A}^{-1}(I) + \mathcal{O}(\hbar)$ defined for, taking values near $\lambda_0$ and solving the equation
\begin{equation}
\mu(\h,\lambda)e^{\frac{i}{\hbar}\mathcal{A}(\lambda)} = e^{iI/\hbar}, ~~ I \in \mathrm{Neigh}(\mathcal{A}(\lambda_0),\C).
\end{equation} 
For that, we will need  a classical analytic symbol $\mathcal{F}(\hbar,\lambda) = \mathcal{A}(\lambda) + \pi \hbar+ \mathcal{O}(\hbar^2)$ such that \[ \mu(\h,\lambda)e^{\frac{i}{\hbar}\mathcal{A}(\lambda)} = e^{\frac{i}{\hbar}\mathcal{F}(\hbar,\lambda)}, \]
and then invert the equation $\mathcal{F}(\hbar,\lambda) = I$. This is the subject of the two following lemmas.

\begin{lemma}
We find a classical analytic symbol $\gamma(\h,\lambda)$ of formal symbol $\sum_{k \geq 0} \h^k \gamma_k$ such that $\mu(\h,\lambda) = e^{i\gamma(\h,\lambda)}$ and $\gamma(\h,\lambda) = \pi + \mathcal{O}(\h)$.
\end{lemma}
\begin{proof}
We apply Lemma \ref{lem:compoholoanaly} with $\gamma(\hbar,\lambda) = i\log(\mu(\hbar,\lambda))$, where we choose the branch defined in a small neighbourhood of $-1$ satisfying $\log(-1) = -i\pi$.
\end{proof}

\begin{lemma}
\label{lem:invanaly}
Suppose that $(\h,\lambda) \longmapsto f(\h,\lambda)$ is a classical analytic symbol defined in a neighbourhood of $\lambda_0$ with formal symbol $\sum_{k \geq 0} \h^k f_k(\lambda)$ satisfying $\partial_{\lambda} f_0(\lambda_0) \neq 0$. Then there exists a classical analytic symbol $g(\h,I)$ of formal symbol $\sum_{k \geq 0} \h^k g_k$ satisfying
\begin{equation}
\exists \mathscr{C}>0, \quad \forall I \in \mathrm{Neigh}(f_0(\lambda_0),\C), \quad f(\h, g(\h,I)) = I + \mathcal{O}(e^{-\frac{\mathscr{C}}{\hbar}}) ,
\end{equation}
and
\begin{equation}
\exists \mathscr{C}>0, \quad \forall \lambda \in \mathrm{Neigh}(\lambda_0,\C), \quad g(\h,f(\h,\lambda)) = \lambda + \mathcal{O}(e^{-\frac{\mathscr{C}}{\hbar}}).
\end{equation}

\end{lemma}

\begin{proof}
In the following proof, we denote formal symbols with majuscules and classical analytic symbols with minuscules. According to this remark, $F(\hbar,\lambda)$ will refer to the formal symbol associated with $f(\hbar,\lambda)$.\\

The asymptotic expansion is obtained by the usual iterative method giving a formal symbol 
$G = \sum_{k \geq 0} \h^k g_k,$ which is the formal inverse of the formal symbol of $f(\hbar,\lambda)$. Of course the $g_k$ are well defined and holomorphic in a small neighbourhood of $\lambda_0$ (independent of $k$). With no loss of generality, in order to simplify the notations, we will suppose that $\lambda_0 = 0$.

\textit{1) Estimate on the formal symbol } \\ 
In their recent article \cite[Proposition 3.5]{hitrik2024} Hitrik and Zworski check by arm and hammer that the formal symbol $G$ is a formal classical analytic symbol. Let us notice that a variant of their proof can be obtained after the reading of Komatsu \cite{Komatsu} (which proves similar results for ultradifferentiable mappings). In our situation, denoting $f_{k,n} = \frac{f_n^{(k)}(0)}{k!}$, the datum of a ($1$-variable) formal classical analytic symbol is equivalent to that of a formal series
\[ \sum_{k,n \in \N^2} f_{k,n} \hbar^n \lambda^k, \quad \exists C,K >0, ~~ |f_{k,n}| \leq C^k K^n n!\,. \]
The key of the proof of Komatsu is that he had the idea of using Lagrange's inversion formula. Let us follow his insight and use (without proof) the following result. 

\begin{lemma}[Lagrange Inversion Formula, \cite{Lagrange} Theorem 2.2.1]
Let $S(\hbar,\lambda) = \sum_{n \geq 1} \widehat{s}_n(\hbar) \lambda^n$ a formal power series where each $\widehat{s}_n := \sum_{k\geq 0} s_{k,n} \hbar^k \in \C[[\hbar]]$ is a formal symbol. Let us assume that $s_{0,1} \neq 0$, and denote $L(\hbar,\lambda) = \lambda/S(\hbar,\lambda)$. The formal series $S(\hbar,\lambda)$ then admits a formal inverse denoted $S^{-1}(\hbar,t)= \sum_{n \geq 1} \widehat{\alpha}_n(\hbar) t^n$, it is given by
\begin{equation}
\label{eq:laginv}
\widehat{\alpha}_j = \frac{1}{j} [L^j]_{j-1} \hbox{ where we write} \left [ \sum_{k \geq 0} \widehat{\beta}_k(\hbar) X^k\right]_j := \widehat{\beta}_j(\hbar). 
\end{equation}
\end{lemma}
We wish to apply the previous lemma with $S(\hbar,\lambda) = F(\hbar,\lambda) - F(\hbar,0)$. We notice that by Lemmas \ref{lem:compoholoanaly} and \ref{lem:prodanal}, $L(\hbar,\lambda) = \lambda/S(\hbar,\lambda)$ is a formal classical analytic symbol. It can be seen as the datum of a formal series $\sum_{(k,n) \in \N^2} \ell_{n,k} \lambda^k \hbar^n$ with $|\ell_{n,k}| \leq C^{k+n} n!$ for some $C>0$. Moreover $L^j(\hbar,\lambda)$ is a formal classical analytic symbol and by Lemma \ref{rk:pwest} we infer that there exists another $C>0$ such that uniformly in $j \in \N$,
$[L^j]_{n,k} \leq (3C)^{j-1} C^{k+n} n!$, where $[\cdot]_{n,k}$ denotes naturally the coefficient in front of $\lambda^k \hbar^n$. From Formula \ref{eq:laginv} we obtain a formal classical analytic symbol $S^{-1}(\hbar,I)$ which is the formal inverse of $F(\hbar,\lambda) - F(\hbar,0)$. In order to conclude, let us recall again that by Proposition \ref{prop:compoanaly} the composition of formal classical analytic symbols is still a formal classical analytic symbol (in the sense of Definition \ref{s0}). We deduce, using this result, that $G(\hbar,I) := S^{-1}(\hbar,I-F(\hbar,0))$ is the formal inverse of $F(\hbar,\lambda)$ and that it is a formal classical analytic symbol.

\textit{2) Conclusion} \\
Let us recall again that by Proposition \ref{prop:compoanaly} the composition of classical analytic symbols is still a classical analytic symbol. Let us then define $g$ as the resummation (in the sense of Proposition \ref{resummation}) of the formal symbol constructed above. Therefore $f(\hbar,g(\hbar,I))$, respectively $g(\hbar,f(\hbar,\lambda))$ are classical analytic symbols whose formal symbols are $I$, respectively $\lambda$. Eventually Proposition \ref{resummation} implies that $f(\hbar,g(\hbar,I))$, respectively $g(\hbar,f(\hbar,\lambda))$ are exponentially close to the identity mapping.
\end{proof}

In the end, for $\lambda$ near $\lambda_0$, with $\mathcal{F}(\h,\lambda) = \mathcal{A}(\lambda) + \h \gamma(\h,\lambda)$ we have
\begin{align*}
\mu(\h,\lambda)e^{i\mathcal{A}(\lambda)/\h}  = 1 & \Longleftrightarrow e^{\frac{i}{^h}(\h\gamma(\h,\lambda)+\mathcal{A}(\lambda))} = 1\Longleftrightarrow e^{\frac{i}{\h}\mathcal{F}(\h,\lambda)} = 1, \\
& \Longleftrightarrow \mathcal{F}(\h,\lambda) \in 2\pi \h\Z \cap \mathrm{Neigh}(\mathcal{A}(\lambda_0),\R).
\end{align*}
As $\partial_{\lambda} \mathcal{A}(\lambda_0) = T(\lambda_0) \neq 0$ therefore applying the local inverse function lemma for analytic symbols yields a classical analytic symbol \begin{equation}
\label{eq.invsymb}
\lambda(\h,I) = \mathcal{A}^{-1}(I - \pi \hbar) + \mathcal{O}(\hbar^2), ~~ I \in \mathrm{Neigh}(\mathcal{A}(\lambda_0),\C) \, ,
\end{equation}
 solving $\lambda(\h,\mathcal{F}(\hbar,\lambda))) = \lambda + \mathcal{O}(e^{-\frac{\mathscr{C}}{\hbar}})$. This yields, denoting with a slight abuse $\lambda(\hbar, 2\pi  \hbar \Z) := \{ \lambda(\h,2\pi k \h) ~ | ~ k \in \Z, ~ 2 \pi \h k \in \mathrm{Neigh}(\mathcal{A}(\lambda_0),\R) \}$, for $\lambda$ near $\lambda_0$ 
\begin{equation}
\begin{aligned}
\mu(\h,\lambda)e^{\tfrac{i}{\hbar}\mathcal{A}(\lambda)/\h}  = 1 & \Longleftrightarrow \mathcal{F}(\h,\lambda) \in 2\pi \h \Z \cap \mathrm{Neigh}(\mathcal{A}(\lambda_0),\R) \\ & \Longleftrightarrow   \lambda + \mathcal{O}(e^{-\frac{\mathscr{C}}{\hbar}}) \in \lambda(\hbar,2\pi \Z \hbar) \cap \mathrm{Neigh}(\lambda_0,\C).
\end{aligned}
\end{equation}
This proves thanks to Lemma \ref{lem:bscond} that, there exist another $\mathscr{C}>0$ and a function $\mathcal{O}(e^{-\frac{\mathscr{C}}{\hbar}})$ such that for all $\lambda \in \mathrm{Neigh}(\lambda_0,\R) \cap \lambda(\hbar,2\pi \Z \hbar)$ there exists an eigenvalue $\lambda_\hbar$ of $P_{\Phi_0}$ satisfying $|\lambda - \lambda_\hbar | \leq \mathcal{O}(e^{-\frac{\mathscr{C}}{\hbar}})$. 

In conclusion, thanks to the $\mathcal{O}(\hbar^2)$ description of the spectrum given by Lemma \ref{lem:bscinfinity}, we infer that Theorem \ref{th1} is achieved in a small neighbourhood of $\lambda_0$. The patching Lemma \ref{lem.patching} allows to extend this result in a small neighbourhood of $(E_1,E_2)$.

\section*{Acknowledgments}
This work was conducted within the France 2030 framework programme, Centre Henri Lebesgue ANR-11-LABX-0020-01 and the author was financed by the AMX grant of \'Ecole polytechnique. The author is very grateful to San Vũ Ng\d{o}c, his main thesis supervisor, for weekly discussions, support in writing, and proofreading of the article. The author is also indebted to Yannick Guedes Bonthonneau, his co-supervisor, for his careful proofreading and many pieces of advice on the writing of the article. The author also would like to warmly thank the anonymous referee at \emph{Journal of Spectral Theory} for his very constructive remarks which allowed to improve the quality of the paper. The author also wishes to thank Michael Hitrik and Martin Vogel for their careful reading of his Ph.D. dissertation, in which an earlier version of this paper appeared, and for their constructive remarks.

\appendix

\section{Analytic symbols and analytic constructions}
\label{apx.1}
In this appendix, we recall the main properties of classical analytic symbols and analytic microlocal analysis developped in \cite{AST_1982__95__R3_0} (see also \cite{RSV} for a recent exposition).
\subsection{Standard Definitions}
Let $\Omega \subset \C^d$ an open set.
\begin{definition}[Formal Classical Analytic Symbols]
\label{defc}
A \emph{formal classical analytic symbol} $\widehat{a}_\h$ on $\Omega$ is a formal series $\widehat{a}_\h = \sum_{n \geq 0} a_n \h ^n$ where each $a_n \in \mathrm{Hol}(\Omega)$ satisfies 
\begin{equation}
\label{eq:analyest}
\forall K \Subset \Omega, ~ \exists C_K > 0, ~ \forall n \in \N, ~~ \underset{K} \sup ~ |a_n| \leq C_K^{n+1} n^n.
\end{equation}
\end{definition}

\begin{remark}[Stirling Equivalent]
The reader should notice that the estimate \eqref{eq:analyest} in the previous definition is, thanks to the Stirling equivalent of $n!$, equivalent to 
\begin{equation}
\label{eq:analyest2}
\forall K \Subset \Omega, ~ \exists C_K > 0, ~ \forall n \in \N, ~~ \underset{K} \sup ~ |a_n| \leq C_K^{n+1} n!.
\end{equation}
\end{remark}

The most natural class of symbols that we can define is the class of \emph{classical symbols} which will be supposed to be holomorphic throughout this paper.

\begin{definition}[Classical Symbols]
The map $a : (0,\h_0) \times \Omega \to \C$ is a classical symbol if for every $\h \in (0,h_0)$, $a(\h,\cdot) \in \mathscr{C}^{\infty}(\Omega)$ and admits a formal classical analytic symbol $\widehat{a}_\h = \sum_{n \geq 0} a_n \h ^n$ as asymptotic expansion in the sense that 
\begin{equation}
\label{maj0}
\forall K \Subset \Omega, ~~ \forall N \in \N, \quad \left | a(\h,\cdot) - \sum_{n = 0}^{N-1} \h ^n a_n \right|_{L^{\infty}(K)} = \mathcal{O}(\h^N).
\end{equation}
\end{definition}

Refining the constraint \eqref{maj0} yields the definition of classical analytic symbols.

\begin{definition}[Classical Analytic Symbols]
\label{s0}
The map $a : (0,\h_0) \times \Omega \to \C$ is a classical analytic symbol if for every $\h \in (0,h_0)$, $a(\h,\cdot) \in \mathrm{Hol}(\Omega)$ and $a$ admits a formal classical analytic symbol $\widehat{a}_\h = \sum_{n \geq 0} a_n \h ^n$ as asymptotic expansion in the sense that 
\begin{equation}
\label{maj}
\forall K \Subset \Omega, ~ \exists C_K > 0,  ~~ \forall N \in \N, \quad \left | a(\h,\cdot) - \sum_{n = 0}^{N-1} \h ^n a_n \right|_{L^{\infty}(K)} \leq \h^N C_K^{N+1}N^N.
\end{equation}
\end{definition}
In fact \eqref{maj} is crucial as it guarantees that classical analytic symbols come with a resummation process which is exponentially sharp. This is key to replace $\mathcal{O}(\hbar^\infty)$ remainders with exponentially sharp ones. The idea is to notice that when $n \leq \frac{1}{\h}$, $\h^n \leq n^{-n}$, cancelling the exponential growth of $|a_n|$.

Sometimes for reasons of convenience of the notation, we denote $a_\h$ instead of $a(\cdot,\cdot)$ the classical analytic symbols. We will also 
say that $a_\h$ is associated with the formal analytic symbol $\sum_{k \geq 0} a_n \hbar^n$.

\begin{proposition}[Resummation]
\label{resummation}
Let $a_\h$ a classical analytic symbol on $\Omega$ of formal symbol $\sum_{n \geq 0} a_n \h^n $, $K \Subset \Omega$ and $C_K$ given by Definition (\ref{defc}) then setting
\begin{equation}
\forall K \Subset \Omega, \quad b_{\h,K} := \sum_{n = 0}^{\lfloor 1/\mathscr{C} \h \rfloor} a_n \h^n, \hbox{ where }  \mathscr{C} \geq C_K e,
\end{equation}
 we have $\sup_{K} |a_\h - b_{\h,K}| \leq C_K e^{-1/\mathscr{C}\h}$.
\end{proposition}
\begin{proof}
Replacing $N$ with $\lfloor \frac{1}{\mathscr{C} \h} \rfloor$ in (\ref{maj}), we have on $K \Subset \Omega$
$$ | a_\h - b_{\h,K}| = \left | a_\h - \sum_{n = 0}^{\lfloor 1/\mathscr{C} \h\rfloor} a_n \h^n \right | \leq C_Ke^{\frac{1}{\mathscr{C}\h}\ln(C_K\h/\mathscr{C}\h)} \leq C_Ke^{-\frac{1}{\mathscr{C}\h}}.$$
\end{proof}

Notice that this procedure allows to distinguish a classical analytic symbol $a(\hbar,z)$ from for instance $a(\hbar,z) + e^{-\frac{1}{\sqrt{\hbar}}}$ which has the same classical formal symbol but does not satisfy \eqref{s0}. From the proof of the previous proposition, it is natural to notice the following.

\begin{lemma}
\label{rk:equiv_resum}
A classical symbol on $\Omega$, holomorphic in $z$, satifies the estimate of Definition \ref{s0}, if and only if it is exponentially close, on each compact of $\Omega$, to a classical analytic symbol with same formal symbol.
\end{lemma}
\begin{proof}
The direct sense is a straightforward consequence of Proposition \ref{resummation}. Let us prove the reciprocal \emph{i.e.} that classical symbols, holomorphic in $z$, of the form $a(\hbar,\cdot)+ \mathcal{O}(e^{-\frac{\mathscr{C}}{\hbar}})$ with $a(\hbar,\cdot)$ a classical analytic symbol and $\mathscr{C}>0$, satisfy the inequality \eqref{maj}. In other words, we want, given $C>0$, to prove that there exists $C'>0$ such that 
\[ \forall \hbar \in (0,\hbar_0), ~ \forall N \in \N, \quad (C \hbar N)^N+ \mathcal{O}(e^{-\frac{\mathscr{C}}{\hbar}}) \leq (C' \hbar N)^N. \]
We have some $r>0$ such that
\[ (C \hbar N)^N+ \mathcal{O}(e^{-\frac{\mathscr{C}}{\hbar}}) \leq (C' \hbar N)^N  \left( \left( \frac{C}{C'} \right)^N + \frac{r e^{-\frac{\mathscr{C}}{\hbar}}}{(C'\hbar N)^N} \right). \]
We want the term $\left( \frac{C}{C'} \right)^N + \frac{r e^{-\frac{\mathscr{C}}{\hbar}}}{(C'\hbar N)^N}$ to be bounded uniformly in $N \in \N$ and $\hbar \in (0,\hbar_0)$. The non-divergence of $\left( \frac{C}{C'} \right)^N$ implies $C'>C$. Let us then estimate
\begin{equation} 
\label{eq:bnd}
\frac{e^{-\frac{\mathscr{C}}{\hbar}}}{(C'\hbar N)^N} = e^{-\frac{1}{\hbar}(\mathscr{C} + N \hbar \ln(C' N \hbar))}.
\end{equation}
to obtain another condition on $C'$. The map $\hbar \longmapsto N \hbar \ln(C' \hbar N)$ attains its minimum at $\hbar=\frac{1}{C' N e}$ and its value is $-\frac{1}{C'e}$. Therefore choosing $C' \geq \frac{1}{\mathscr{C}e}$ is sufficient to guarantee the boundedness of \eqref{eq:bnd}. Eventually choosing $C' = \max(C,\frac{1}{\mathscr{C}e})$ seals the proof.
\end{proof}

Let us now present a useful patching result which is independent of the other upcoming results. The set $\mathcal{F}$ defined on the following lemma can be viewed for instance as the disjoint union $\displaystyle \underset{\hbar \in (0,\hbar_0)}{\bigsqcup}\sigma_\hbar(P_{\Phi_0})$.

\begin{lemma}[Patching Classical Analytic Symbols on a Grid]
\label{lem.patching}
Let $a_1(\hbar,\lambda),a_2(\hbar,\lambda)$ two classical analytic symbols (of the variable $\lambda$) with formal symbol $\sum_{n\geq 0} a^n_1(\lambda)\hbar^n$ and $\sum_{n\geq 0}a^n_2(\lambda)\hbar^n$ and defined on some connected open sets $\Omega_1, \Omega_2$ with $\Omega_1 \cap \Omega_2 \neq \emptyset$. Suppose that there exists a set $\mathcal{F} \subset (0,\hbar_0) \times \bigl(\Omega_1 \cap \Omega_2\bigr)$ satisfying the following.

---There exists an infinite subset $\mathcal{S} \subset \Omega_1 \cap \Omega_2$ such that for every $\lambda \in \mathcal{S}$, there exists a sequence $(\hbar_j,\lambda_j)_{j \in \N}$ in $\mathcal{F}$ converging towards $(0,\lambda)$.

---There exists $r,\epsilon_0 > 0$
\[ \forall (\hbar,\lambda) \in  \mathcal{F}, ~~ |a_1(\hbar,\lambda) - a_2(\hbar,\lambda)| \leq re^{-\epsilon_0/\hbar}. \]
Then for all $K \Subset \Omega_1 \cup \Omega_2$, we find a glued classical analytic symbol $a_{K}(\hbar,\cdot)$ and $C,\epsilon > 0$ such that for $i=1,2$,
\[\forall K_i \Subset \Omega_i, \quad |a_i(\hbar,\cdot) - a_{K}(\hbar,\cdot)| \leq Ce^{-\epsilon/\hbar} \hbox{ on } \Omega_i \cap K \cap K_i. \]
\end{lemma}
\begin{proof}
For all $\lambda \in \mathcal{S}$, we have a sequence $(\hbar_j,\lambda_j)_{j \in \N}$ in $\mathcal{F}$ converging towards $(0,\lambda)$. This gives the recursion formula
\[a_1^n(\lambda_j) - a_2^n(\lambda_j) = \mathcal{O}(\hbar_j) \xrightarrow[\hbar_j \to 0]{} 0 \hbox{ i.e } a_1^n = a_2^n \hbox{ on } \mathcal{S}.\]
This proves that for all $n \in \N$ and $\lambda \in \mathcal{S}$, $a_1^n(\lambda) = a_2^n(\lambda)$ so that by holomorphy they coincide on $\Omega_1 \cap \Omega_2$. This provides a formal analytic symbol on $\Omega_1 \cup \Omega_2$ coinciding with the formal analytic symbols of $a_1(\hbar,\cdot)$ and $a_2(\hbar,\cdot)$ on their domain. Applying the resummation Lemma \ref{resummation} seals the proof.
 \end{proof}

We now present some standard and useful results about classical analytic symbols.

\subsection{Composition of classical analytic symbol}
\subsubsection{Statement of the result}
In all that follows, it will be implied that the formal symbols associated with classical analytic symbols $f(\hbar,\lambda), g(\hbar,I)$ are $\sum_{n \geq 0} f_n(\lambda) \hbar^n$ and $\sum_{n \geq 0} g_n(I) \hbar^n$. The goal of this section is to prove that we can compose (in the sense of functions) classical analytic symbols.

\begin{proposition}
\label{prop:compoanaly}
Let $f(\hbar,\lambda)$ (defined near $\lambda_0$) and $g(\hbar,I)$ (defined near $f_0(\lambda_0)$) be two classical analytic symbols, then $g(\hbar,f(\hbar,\lambda))$ is a classical analytic symbol near $\lambda_0$.
\end{proposition}

\subsubsection{Proof of the result}
To prove this result we must go through some useful tools and results. The first ones appear when trying to prove the following lemma.

\begin{lemma}
\label{lem:prodanal}
Let $a(\hbar,\lambda)$, $b(\hbar,\lambda)$ be two classical analytic symbols defined near $\lambda_0$, then $a(\hbar,\lambda)b(\hbar,\lambda)$ is a classical analytic symbol.
\end{lemma}
\begin{proof}
In what follows, we make a simplifying assumption to lighten the exposition. We will consider a compact domain $\overline{\Omega}$ where both symbols are well defined and assume that the estimates \eqref{eq:analyest} and \eqref{s0} are satisfied with $K = \overline{\Omega}$. We denote $|a_k| = |a_k|_{L^{\infty}(\Omega)}$ similarly for $b_k$. To prove the result, we check that \eqref{eq:analyest} and \eqref{s0} are satisfied.

\textit{1) Estimate on the formal symbol} \\
The idea is to use powerful results on converging power series. Let us notice that for $\rho >0$ small enough, given a formal classical analytic symbol $D(\hbar,\lambda) = \sum_{k\geq 0} d_k(\lambda) \hbar^k$ the series given by the Borel transform
$\sum_{n \in \N} \frac{|d_k|}{k!} \rho^k$ converges for $\rho$ small enough. This is interpreted as a \emph{pseudonorm} on the classical analytic symbol in the sense that a formal classical symbol is analytic if and only if there exists $\rho>0$ such that this sum is converging. It is often denoted $\| D \|_\rho$ in the litterature. A great advantage is that this pseudonorm is multiplicative and additive,
\[\forall \rho >0, \quad \|A B\|_{\rho} \leq \|A \|_{\rho} \|B\|_{\rho} \quad \|A+ B\|_{\rho} \leq \|A \|_{\rho} + \|B\|_{\rho} . \]
This proves that the formal symbol $AB$ of $a(\hbar,\lambda)b(\hbar,\lambda)$ is a classical analytic formal symbol. 

\textit{2) Comparison between the symbol and the formal symbol} \\
This is slightly more subtle and scarcely adressed in the litterature. Let us write $d(\hbar,\lambda) = \sum_{m \in \N} d_m \hbar^m$ with $d_m = \sum_{k=0}^m a_k b_{m-k}$ the formal symbol of $a(\hbar,\lambda)b(\hbar,\lambda)$. We notice denoting $S_n(a) = \sum_{k=0}^n \hbar^k b_k$, $S_n(b) = \sum_{k=0}^n \hbar^k b_k$
\[\sum_{m=0}^n d_m \hbar^m = \sum_{k=0}^n  \hbar^k a_k  S_{n-k}(b).\]
Let us also denote\footnote{Notice that $\delta_n(a), \delta_n(b)$ depend on $\hbar$.} $\delta_k(a) = a - S_k(a)$ and similarly for $\delta_k(b)$. We obtain
\[\sum_{m=0}^n d_m \hbar^m = \sum_{k=0}^n  \hbar^k a_k (b-\delta_{n-k}(b))= (a - \delta_n(a))b -  \sum_{k=0}^n  \hbar^k a_k \delta_{n-k}(b).\]
This proves the formula 
\begin{equation}
\label{eq:multnorm}
\delta_{n}(ab) = \delta_n(a)b + \sum_{k=0}^n  \hbar^k a_k \delta_{n-k}(b).
\end{equation}
From \eqref{s0} we know that 
$ \displaystyle \ell_n(a) := \sup_{\hbar \in (0,\hbar_0)} \frac{\delta_n(a)}{\hbar^n} < + \infty$ which allows to define 
\[ \alpha_0(a) = \max\left(|a_0|, \sup_{\hbar \in (0,\hbar_0)} |a|\right) \hbox{ and }\alpha_n(a) = \max(|a_n|,\ell_{n-1}(a))\] for $n \geq 1$, and similarly for $b$. This allows to define a pseudonorm \footnote{Which is distinct from the pseudonorm on formal classical analytic symbols.} on classical analytic symbols using once again the Borel transform 
\[ \| d(\hbar,\cdot) \|_\rho = \sum_{k=0}^n \frac{\alpha_k(d)}{k!} \rho^k.\]
The finiteness of this new norm is therefore equivalent to the fact that the symbol satisfies simultaneously \eqref{eq:analyest} (which is already known) and \eqref{s0}. Thanks to \eqref{eq:multnorm} this new pseudonorm is also multiplicative \emph{i.e.} letting $a(\hbar,\lambda),b(\hbar,\lambda)$ two classical analytic symbols
\[\| a(\hbar,\cdot) b(\hbar,\cdot)\|_{\rho} \leq \| a(\hbar,\cdot)\|_{\rho}\| b(\hbar,\cdot)\|_{\rho} \, .\]
\end{proof}

The next step is then to prove that the composition of classical analytic symbols with an holomorphic function is still a classical analytic symbol.

\begin{lemma}
\label{lem:compoholoanaly}
Let $f(\hbar,\lambda)$ a classical analytic symbol near $\lambda_0$ and $g$ in \\ $\mathrm{Hol}(\mathrm{Neigh}(f_0(\lambda_0),\C))$ then the composed classical symbol $g(f(\hbar,\lambda))$ is a classical analytic symbol in a small neighbourhood of $\lambda_0$.
\end{lemma}
\begin{proof}
We will use the norm $\| \cdot \|_{\rho}$ defined in the proof of the previous lemma. Let us notice, by elementary results of summable families that if $a(\hbar,\lambda) = \sum_{k \geq 0} a_k(\hbar,\lambda)$ is an infinite sum of classical symbols (whose formal symbol converges formally)
\[\| a(\hbar,\cdot) \|_\rho \leq \sum_{k = 0}^{+ \infty} \|a_k(\hbar,\cdot)\|_\rho, \]
where by convention $\|a(\hbar,\lambda)\|_{\rho} = + \infty$ if the series $\sum_{n \geq 0} \tfrac{a_n}{n!} \rho^n$ does not converge. We apply this to the estimate, letting $\rho>0$ small enough and assuming $f_0(\lambda_0) = 0$ (up to replacing $g$ with $g(\cdot - f_0(\lambda_0))$,)
\[ \| g(f(\hbar,\lambda)) \|_{\rho} \leq \sum_{n=0}^{+\infty} \frac{|g^{(n)}(0)|}{n!} \| f(\hbar,\lambda) \|_{\rho}^n. \]
Let us now precise some points. We can suppose that $g$ is defined on a disc $\mathbb{D}(0,r)$ of radius $r>0$ small enough. Similarly, we can choose the domain of $f(\hbar,\lambda)$ to be another disc $\mathbb{D}(\lambda_0,\widetilde{r})$ so that $|f_0|_{L^{\infty}(\mathbb{D}(\lambda_0,\tilde{r}))} \leq r/2$.

 Cauchy's estimates yields  $\displaystyle \frac{|g^{(n)}(0)|}{n!} \leq r^{-n}|g|_{L^{\infty}(\mathbb{D}(0,r))}.$ We infer
\begin{equation}
\label{eq:compoholo}
 \| g(f(\hbar,\lambda)) \|_{\rho} \leq  |g|_{L^{\infty}(\mathbb{D}(0,r))} \sum_{n=0}^{+\infty} \left(\frac{\| f(\hbar,\lambda) \|_{\rho}}{r}\right)^{n} \leq \frac{|g|_{L^{\infty}(\mathbb{D}(0,r))}}{1-(\| f(\hbar,\lambda) \|_{\rho}/r)}. 
 \end{equation}
From the fact that $\| f(\hbar,\lambda)\|_\rho \xrightarrow[\rho \to 0^+]{} |f_0|_{L^{\infty}(\mathbb{D}(\lambda_0,\widetilde{r}))}$, we infer that for $\rho>0$ small enough $\| g(f(\hbar,\lambda)) \|_{\rho}$ is finite proving that $g(f(\hbar,\lambda))$ is a classical analytic symbol.
\end{proof}

To conclude the proof of Proposition \ref{prop:compoanaly}, we write
\[\| g(\hbar,f(\hbar,\lambda)\|_\rho \leq \sum_{n=0}^{+ \infty} \| g_k(f(\hbar,\lambda)) \|_{\rho} \frac{\rho^k}{k!} \leq \frac{1}{1-\| f(\hbar,\lambda) \|_{\rho}/r} \sum_{k=0}^{+\infty} \frac{| g_k |_{L^{\infty}(\mathbb{D}(0,r))}}{k!}. \]
We infer that choosing a small enough neighbourhood of $\lambda$, $\| g(\hbar,f(\hbar,\lambda))\|_\rho < + \infty$ for $\rho>0$ small enough which seals the proof.

\subsubsection{Other useful estimates}
Let $a(\hbar,\lambda)$ be a classical analytic symbol. In this section we provide some estimates on the formal symbol of $a(\hbar,\lambda)^j$. The proof of Proposition \ref{prop:compoanaly} gives that for all $j \in \N$, the power series $\sum_{k \geq 0} \frac{\rho^k}{k!} |a_k^j|_{L^\infty}$ have the same radius of convergence $R_a$ than $\sum_{k \geq 0} \frac{\rho^k}{k!} |a_k|_{L^\infty}$. This allows to infer easily that
$a_k^j = o_{k \to + \infty}(k!(R_a+\epsilon)^k)$ for all $\epsilon>0$. In the following lemma, we provide a slightly more explicit estimate which will be useful for instance in Lemma \ref{lem:invanaly}.

\begin{lemma}
\label{rk:pwest}
Let $A(\hbar,\lambda)$ a formal classical analytic symbol defined on an open set $\Omega$, then denoting $A^j(\hbar,\lambda) = \sum_{n \geq 0} a_n^j(\lambda) \hbar^n$,
\begin{equation}
\forall K \Subset \Omega, \quad \exists C>0, ~ \forall (j,n) \in \N^* \times \N, ~ |a_n^j| \leq (3C)^{j-1} C^{n+1} n! \hbox{ on } K.
\end{equation}
\end{lemma}
\begin{proof}
Let $a(\hbar,\lambda)$, $b(\hbar,\lambda)$ be classical analytic symbols,
\[\exists C,\alpha,\beta>0, \quad |a_k| \leq \alpha C^k k! \hbox{ and } |b_k| \leq \beta C^k k! \hbox{ on } K.\]
If we denote $\sum_{m \geq 0} d_m \hbar^m$ the formal symbol of $a(\hbar,\lambda)b(\hbar,\lambda)$ let us prove that we have
\begin{equation}
|d_m| \leq 3 \alpha \beta C^m m!. 
\end{equation}
 In order no to overload the notations, the dependency in $\lambda$ will be implied. For all $m \in \N$, from the formula of the product of power series, we infer
\begin{equation}
\label{eq:estprodapx}
 | d_m| \leq \sum_{k=0}^m \alpha C^k k! \beta C^{m-k} (n-k)! = \alpha \beta C^m m! \sum_{k=0}^m \binom{m}{k}^{-1}.
 \end{equation}
Let us define $S_m = \sum_{k=0}^m \frac{1}{\binom{m}{k}}$. Letting $k \in [\![2,m-2]\!]$ we have $\binom{m}{k} \geq \binom{m}{2}$ thus $|S_m -2| \leq \frac{2}{m}$ for $m \geq 4$ and computing $S_0=1, S_1=2,S_2=5/2, S_3 = 8/3$ we deduce that $S_m \leq 3$ for all $m \in \N$ and $|d_m| \leq 3 \alpha \beta C^m m!$.

Therefore by a straightforward recursion using the estimate \eqref{eq:estprodapx} the $n$-th term of $a(\hbar,\lambda)^j$ is bounded by $(3C)^{j-1} C^{n+1} n!$.
\end{proof}

\subsection{Pseudodifferential operators with classical analytic symbols}
\label{apx:pseudocxe}
Analytic symbols are useful for computing WKB expansions, which are functions of the form $a_\h e^{i\varphi/\h}$ where $\varphi$ solves some eikonal equation and $a_\h$ is a classical analytic symbol. In \cite{AST_1982__95__R3_0}, Sjöstrand has developped a functional framework of useful spaces of weighted holomorphic functions for finding and manipulating WKB expansions with exponential sharpness.

\subsubsection{Useful function spaces}

\begin{definition}[Space $H_{\Phi}^{\mathrm{\mathrm{loc}}}(\Omega)$]
Let $\Omega$ an open subset of $\C^d$. Let $\Phi \in \mathscr{C}^\infty(\Omega,\R)$ and $u =(u_\h)_{\h \in  (0,\h_0)}$ a family of functions indexed by $\h$, then we say that $u \in H_{\Phi}^{\mathrm{loc}}(\Omega)$ if
\begin{enumerate}
\item $\forall \h \in (0,\h_0), ~~ u_\h \in \mathrm{Hol}(\Omega)$;
\item $\forall K \Subset \Omega, ~ \forall \ep > 0, ~ \exists C_K>0,$ such that $|u_\h| \leq C_K e^{(\Phi+\ep)/\h}$ on $K$ for all $\h \in (0,\h_0)$.
\end{enumerate}

\end{definition}

This allows to define the space of germs $H_{\Phi,x_0}$ at a given point $x_0 \in \C^d$ by the inductive limit
\begin{equation}
H_{\Phi,x_0} := \underset{\substack{\longrightarrow\\ \Omega \ni x_0}} \lim H_{\Phi}^{\mathrm{loc}}(\Omega),
\end{equation}
where $\Omega$ varies in the set of open neighbourhoods of $x_0$. We can quotient these spaces by considering exponentially small germs as negligible.

\begin{definition}
An element $u_\h \in H_{\Phi,x_0}$ is called \emph{negligible} if it belongs to the space 
\begin{equation}
\mathcal{N} := \{ u_\h \in H_{\Phi,x_0} ~ | ~ \exists c >0, ~ \exists \Omega \hbox{ open neighbourhood of }x_0, ~ u_\h \in H_{\Phi-c,x_0}^{\mathrm{loc}}(\Omega) \},
\end{equation}
allowing to define a quotient space $H_{\Phi,x_0}/\mathcal{N}$. 
\end{definition}

We will still denote by $H_{\Phi,x_0}$ the quotient space not to overload the notation (and when the point of reference is clear) we simply write $H_{\Phi}$ and we prefer to denote by $\sim$ the equality in $H_{\Phi}$ spaces. If $u_\h \in H_{0,x_0}$, we say that $u_\h$ is an analytic symbol and by a simple estimate one can see that classical analytic symbols are indeed analytic symbols. 

\subsubsection{Defining complex pseudodifferential operators}
\label{sec:pseudo1}
We define complex pseudodifferential operators on the spaces of germs $H_{\Phi}$ and recall results which can be found in \cite[Chapter 3]{RSV}. We will rather choose to write the operators in the Weyl's quantization, however the interested reader could consult \cite[Section 2]{RSV} providing a good description of change of quantization in this ``analytic'' context.

Let us start with the preliminary definition of ``good contours of integration'' (in the bounded and local case). In the following definition, $\psi \in \mathscr{C}^{\infty}(\mathrm{Neigh}((x_0,y_0),\C^{d+m}))$ satisfies 
that $\psi(x,\cdot)$ has a unique critical point defining a smooth map \[ y_c : x \in \mathrm{Neigh}(x_0,\C^m) \longrightarrow y_c(x) \in \mathrm{Neigh}(y_0,\C^d). \] 

\begin{definition}[Smooth Family of Bounded Good Contours]
\label{def.famgc}
A smooth family of bounded good contours $\mathrm{Neigh}(x_0,\C^m) \ni x \longmapsto \Gamma(x)$ associated with $\psi$ is by definition a \textit{smooth family of bounded contours} of integration satisfying the \textit{good contour property}.
\begin{enumerate}
\item[$i)$] $\mathrm{Neigh}(x_0,\C^m) \ni x \longmapsto \Gamma(x)$ is a \textit{smooth family of contours} if by definition it
is a family of smooth manifolds (with boundary) of real dimension $d$ such that there exists a smooth map
\begin{equation}
\label{def:F}
F : \left \{ \begin{array}{ccc}
(x,\sigma) & \longmapsto & F(x,\sigma) \\
\mathrm{Neigh}(x_0,\C^m) \times [-1,1]^{d} & \longrightarrow  & \C^{d}
\end{array} \right.\, ,
\end{equation}
$F(x,\cdot) : [-1,1]^d \longrightarrow \Gamma(x)$ is a diffeomorphism parametrizing $\Gamma$ by 
\[ \forall x \in \mathrm{Neigh}(x_0,\C^m), ~~ \Gamma(x) := \{ F(x,\sigma) ~ | ~ \sigma \in [-1,1]^{d} \} \hbox{ and } F(x,0) = y_c(x).\] 

\item[$ii)$] The \textit{good contour property} of the family $x \longmapsto \Gamma(x)$ is the further requirement that there exists $C_{\Gamma}>0$ such that for all $x \in \mathrm{Neigh}(x_0,\C^m)$, for all $y \in \Gamma(x)$, 
\begin{equation}
 -\frac{1}{C_\Gamma} |y-y_c(x)|^2 \leq  \psi(x,y) - \psi(x,y_c(x)) \leq -C_\Gamma |y-y_c(x)|^2.
\end{equation}
\end{enumerate}
\end{definition}

To lighten, we often write ``good contour of integration'' instead of smooth family of bounded good contours. Let $\Phi \in \mathscr{C}^{\infty}(\mathrm{Neigh}(x_0,\C^d))$ subharmonic and $q_\h(z,\zeta)$ a classical analytic symbol defined near $(x_0,\xi_0) := \left(x_0, \tfrac{2}{i} \partial_x \Phi(x_0) \right)$. Using the good contour $$\mathscr{G}_{c,\delta}(x) := \left \{ (y,\zeta) ~ | ~ |x-y| < \delta, ~ \zeta = \tfrac{2}{i} \partial_x \Phi(\tfrac{x+y}{2}) + i c \overline{x-y} \right \}$$
one can define the Weyl's quantization in the $H_{\Phi,x_0}$ space formally by
\begin{equation}
\label{sdo}
\forall u_\h \in H_{\Phi,x_0}, \quad \op^w_\h(q_\h) u_\h = \frac{1}{2 \pi \h} \iint_{\mathscr{G}_{c,\delta}(x)} e^{i \zeta(x-y)/ \h} q_\h(\tfrac{x+y}{2},\xi)u_\h(y)  \mathrm{d}\xi \wedge \mathrm{d}y.
\end{equation}
Similarly we can define the left quantization by
\begin{equation}
\label{sdo2}
\forall u_\h \in H_{\Phi,x_0}, \quad \op^\ell_\h(q_\h) u_\h = \frac{1}{2 \pi \h} \iint_{\mathscr{G}_{c,\delta}(x)} e^{i \xi(x-y)/ \h} q_\h(x,\xi)u_\h(y) \mathrm{d}\xi \wedge \mathrm{d}y .
\end{equation}
with similar contours. Indeed, let us consider $\mathrm{Neigh}(x_0,\C^d)$ small neighbourhoods of $x_0$, and $u_\hbar \in H_{\Phi}(\mathrm{Neigh}(x_0,\C))$. Then up to restricting to $\delta$ small enough, for different $\delta$, the integrals 
\begin{equation}
\label{sdo4}
\frac{1}{2 \pi \h} \iint_{\mathscr{G}_{c,\delta}(x)} e^{i \xi(x-y)/ \h} q_\h(x,\xi)u_\h(y) \mathrm{d}\xi \wedge \mathrm{d}y 
\end{equation}
define the same function near $x_0$ modulo an exponentially small remainder, this phenomenon is an exemple of a more general one referred as \emph{equivalence of good contours}. This proves that formulas \eqref{sdo} and \eqref{sdo2} define operators on $H_{\Phi,x_0} \to H_{\Phi,x_0}$. Notice also that it is possible to integrate on ``fewer variables'' and the same result holds (replace $(x,y,\xi) \in \C^d \times \C^d \times \C^d$ in \ref{sdo4} with $(x,y',\xi') \in \C^d \times \C^m \times \C^m$ for some $m \leq d$). The following remark, in the case when the gauge $\Phi$ is harmonic, is a well known consequence of \cite[Chapters 1,2]{AST_1982__95__R3_0} and will be useful. 

\begin{lemma}
\label{lem:conjanaly}
Let $\Phi = \mathrm{Re}(i\varphi)$ be harmonic and $Q : H_{\Phi,x_0} \to H_{\Phi,x_0}$ be a complex pseudodifferential operator, then $Q_\varphi := e^{-\frac{i}{\hbar} \varphi} Q e^{\frac{i}{\hbar}\varphi}$ defines a complex pseudodifferential operator on $H_{0,x_0}$.
\end{lemma}
\begin{proof}
To simply provide the idea of the proof, this stems from the Kuranishi trick (see \cite[Chapter 9.2]{Zworski}) and equivalence of good contours.
\end{proof}

The following lemma underlines the link between the analytic stationary phase and complex pseudodifferential operators. 

\begin{lemma}
In the case where $\Phi = 0$, the operators $\op_\h^w(q_\h),\op_\h^\ell(q_\h) : H_{0,x_0} \longrightarrow H_{0,x_0}$ are well defined and send classical analytic symbols towards classical analytic symbols. 
\end{lemma}

The action of $\op_\h^w(q_\h),\op_\h^\ell(q_\h)$ at the level of formal symbols is in fact given by that of a \emph{formal pseudodifferential operator} which is an infinite sum of monomial differential operators. In the case of the left quantization the formal pseudodifferential operator is
\[ q_\hbar(z,\hbar D_z) = \sum_{k \geq 0} \frac{\partial_\zeta^{(k)}q_\hbar(z,0)}{k!} (\hbar D_z)^k,\]
while in the case of Weyl's quantization it is $q_{\hbar,\frac12}(z,\hbar D_z)$ with 
\[ q_{\hbar,\frac12}(z,\zeta) := \exp\left( -\tfrac{i}{2\hbar}(\hbar D_\zeta\cdot(\hbar D_z -\hbar D_w))\right) q_{\hbar}|_{w=z}.\]

\begin{proof}
Let us provide some elements of proof, we start with the following useful remark.

\textit{1) The operators send classical analytic symbols towards classical analytic symbols.}

 We let the reader convince himself that thanks to the holomorphic Morse lemma and by separation of variables in the integral, it is sufficient to prove the following.  For all a classical analytic symbol $a(\hbar,z,w)$ defined near some $(z_0,z_0)$, the function
\[b(\hbar,w) = (2\pi \hbar)^{-1/2}\int_{\mathrm{Neigh}(0,\R)} e^{-\frac{z^2}{2\hbar}}a(\hbar,z,w)\mathrm{d}z, \]
is a classical analytic symbol near $0$. 

--- We check that the formal symbol of $b(\hbar,z)$ is a formal classical analytic symbol. Let us denote the formal symbols with majuscules. We obtain from \cite[Theorem 2.1]{AST_1982__95__R3_0}
\[ B(\hbar,w) = \sum_{k\geq 0} \hbar^{k} \sum_{j \geq 0} \hbar^j \frac{\Delta_z^k}{2^k k!} a_j(z,w)|_{z=0} = \sum_{m \geq 0} \hbar^m  \left(  \sum_{\nu = 0}^{m}\frac{\Delta_z^\nu}{2^\nu \nu!}a_{m-\nu}(z,w)|_{z=0}\right).\]
In view of Lemma \ref{lem:prodanal}, we wish to estimate $\| B(\hbar,w)\|_{\rho}$ and we do as follows. Formally, $B$ is obtained by applying the formal sum of differential operators $\mathcal{I}(z,\hbar D_z) := \sum_{k\geq 0} \hbar^k \frac{\Delta_z^k}{2^k k!},$ to a formal classical analytic symbol. This suggests to define a converging series of differential operators by the Borel transform
\begin{equation}
[\mathcal{I}(z,\hbar D_z)]_\rho = \sum_{k=0}^{+\infty} \frac{\Delta_z^k}{2^k k!} \frac{\rho^k}{k!}.
\end{equation}
 Letting $\mathbb{D}(s) \Subset \mathbb{D}(t) \subset \mathbb{D}(t_0) \subset \C$ be small and non empty discs on which $a(\hbar,z,w)$ is well defined for $w$ close enough to $0$. We then infer that the operator $\frac{\Delta_z^\nu}{2^\nu \nu!}$ is a bounded operator $\mathrm{Hol}(\mathbb{D}(t)) \to \mathrm{Hol}(\mathbb{D}(s))$ with norm $\|\frac{\Delta_z^\nu}{2^\nu \nu!} \|_{s,t}$ lower than $C^\nu \nu ! (t-s)^{-2\nu}$ for some $C>0$ (independent of $\nu$). This allows to extend the norm $\|\cdot \|_{\rho}$ on classical analytic symbol to series of differential operators
\[ \mathcal{N}_k = \sup_{0 < s<t\leq t_0} \left \|\frac{\Delta_z^\nu}{2^\nu \nu!} \right \|_{s,t}(t-s)^{2\nu}, \quad \| \mathcal{I} \|_{\rho} = \sum_{n=0}^{+ \infty} \frac{\mathcal{N}_k}{k!} \rho^k, \]
and notice that $\| \mathcal{I} \|_{\rho}$ converges for $\rho$ small enough. We then infer that formally, with $A(\hbar,z,w)$ the classical analytic symbol associated with $a(\hbar,\cdot,\cdot)$,
\[\| \mathcal{I}(z,\hbar D_z) A(\hbar,\cdot,\cdot)\|_\rho \leq \| \mathcal{I} \|_{\rho} \| A(\hbar,\cdot,\cdot) \|_{\rho}. \]
This yields $\| B(\hbar,w)\|_\rho \leq \| \mathcal{I} \|_{\rho} \| A(\hbar,\cdot,\cdot) \|_{\rho} < + \infty$ if $\rho>0$ is small enough. We conclude that $B(\hbar,w)$ is a formal classical analytic symbol. \\
--- It remains to check that the new symbol satisfies the estimate \eqref{maj}. Let $N \in \N$ and write $a(\hbar,z,w) = a_{N-1}(\hbar,z,w) + \delta_N(\hbar,z,w)$, with $a_{N-1}$ the partial sum to the order $\hbar^{N-1}$. We obtain
\begin{equation}
\begin{split}
\int_{\mathrm{Neigh}(z_0,\C)} e^{-\frac{z^2}{2 \hbar}} a(\hbar,z,w)\mathrm{d}z = \int_{\mathrm{Neigh}(z_0,\C)} e^{-\frac{z^2}{2 \hbar}} a_{N-1}(\hbar,z,w)\mathrm{d}z \\ + \int_{\mathrm{Neigh}(z_0,\C)} e^{-\frac{z^2}{2 \hbar}} \delta_N(\hbar,z,w)\mathrm{d}z,
\end{split}
\end{equation}
and the second term that we denote $\widetilde{\delta}_N$ satisfies $|\widetilde{\delta}_N| \leq C_0 C^{N+1} N! \hbar^N$. Let us also denote $\mathcal{I}_m = \sum_{k=0}^m \hbar^k \frac{\Delta_z^k}{2^k k!}$ and recall that by the analytic stationary phase \cite[Theorem 2.1]{AST_1982__95__R3_0}
\[\forall n \in \N, \quad \int_{\mathrm{Neigh}(z_0,\C)} e^{-\frac{z^2}{2 \hbar}} a_{n}(z,w)\mathrm{d}z = \mathcal{I}_m a_n(z,w)|_{z=w} + R_{n,m}(\hbar,w) \] where  $|R_{m,n}(\hbar,w)| \leq (m+1)! C_1^{m+1}\hbar^{m+1} |a_n|_{\infty}$
and $C_1>0$ does not depend on $n$. We obtain
\begin{equation}
\begin{aligned}
\int_{\mathrm{Neigh}(z_0,\C)} e^{-\frac{z^2}{2 \hbar}} a(\hbar,z,w)\mathrm{d}z  & = \sum_{k = 0}^{N-1} \hbar^k \left(\mathcal{I}_{N-1-k} a_k|_{z=w} +R_{N-1-k,k}(\hbar,w)\right) + \widetilde{\delta}_N\\
& = \sum_{k=0}^{N-1} b_k(w) \hbar^k + \sum_{n = 0}^{N-1} \hbar^k R_{N-1-k,k}(\hbar,w)+ \widetilde{\delta}_N.
\end{aligned}
\end{equation}
It is then easy to see that $\eqref{maj}$ is satisfied.

\textit{2) The action of $\op_\h^w(q_\h),\op_\h^\ell(q_\h)$ on formal symbols is given by that of the associated formal pseudodifferential operators.}

The formal formula for $\op_\h^\ell(q_\h)$ is a direct consequence of the analytic stationary phase. The formal formula for $\op_\h^w(q_\h)$ can be also inferred from the analytic stationary phase but it is more easily obtained thanks to the formula for changing between quantization (from the left quantization to the Weyl quantization) that is for instance detailed in \cite[Chapter 3]{RSV}.
\end{proof}

Let us then finish this presentation with recalling that the pseudonorms $\| \cdot \|_\rho$ (that we used to prove the previous lemma) defined on infinite series of differential operators is the main tool that allowed in \cite[Theorem 2.8.1]{minicourse} to prove the existence of the analytic WKB expansions we use in Section \ref{sec:micro}. These pseudonorms are also crucial to prove the existence of parametrix of elliptic pseudodifferential operators with classical analytic symbols (see \cite[Theorem 1.5]{AST_1982__95__R3_0} or \cite[Theorem 2.2.3]{minicourse}).

\printbibliography

\end{document}